\documentclass[12pt,reqno]{amsart}
\usepackage{amsmath}
\usepackage{amsfonts}
\usepackage[colorlinks=true]{hyperref}
\hypersetup{urlcolor=blue, citecolor=red}
\usepackage{geometry}
\usepackage{bbm}
\usepackage{xcomment}
\usepackage{parskip}
\usepackage{paralist}
\usepackage{url}
\usepackage{hyperref}
\newtheorem{theorem}{Theorem}
\numberwithin{theorem}{section}
\theoremstyle{plain}

\newtheorem{corollary}{Corollary}
\numberwithin{corollary}{section}

\newtheorem{definition}{Definition}

\newtheorem{lemma}{Lemma}
\numberwithin{lemma}{section}

\newtheorem{proposition}{Proposition}
\numberwithin{proposition}{section}
\newtheorem{remark}{Remark}
\numberwithin{remark}{section}

\numberwithin{equation}{section}
\def \e {\varepsilon}
\def \p {\partial}
\def\p{\partial}

\newcommand\restr[2]{{% we make the whole thing an ordinary symbol
  \left.\kern-\nulldelimiterspace % automatically resize the bar with \right
  #1 % the function
  \vphantom{\big|} % pretend it's a little taller at normal size
  \right|_{#2} % this is the delimiter
  }}

\begin{document}
\title[Cloaking for a quasi-linear elliptic partial differential equation]{Cloaking for a quasi-linear elliptic partial differential equation}
\author[Tuhin Ghosh]{Tuhin Ghosh}
\address{Jockey Club Institute for Advanced Study, HKUST, Hong Kong}
\email{iasghosh@ust.hk}

\author[Karthik Iyer]{Karthik Iyer}
\address{Department of Mathematics, University of Washington, Seattle, WA 98195-4350, USA}
\email{karthik2@uw.edu}
\urladdr{}
\date{}
\begin{abstract}
In this article we  consider cloaking for a quasi-linear elliptic partial differential equation of divergence type defined on a bounded domain in $\mathbb{R}^N$ for $N=2,3$. We show that a perfect cloak can be obtained via a singular change of variables scheme and an approximate  cloak can be achieved via a regular change of variables scheme. These approximate cloaks though non-degenerate are anisotropic.  We also show, within the framework of  homogenization, that it is possible to get isotropic regular approximate cloaks. This work generalizes to quasi-linear settings previous work on cloaking in the context of Electrical Impedance Tomography for the conductivity equation.
\end{abstract}

\maketitle

\section{Introduction and Preliminaries}

The topic of cloaking has long been fascinating and has recently attracted a lot of attention within the mathematical and broader scientific community. A 
region of space is said to be cloaked if its contents along with the existence of such a cloak are invisible to wave detection. One particular route to cloaking that has received considerable interest is that of transformation optics, 
the use of changes of variables
to produce novel optical effects on waves or to facilitate computations.

A transformational optics approach to 
cloaking using the invariance properties of the conductivity equation was first discovered
by Greenleaf, Lassas and Uhlmann \cite{GLU,GLU1} in 2003. Pendry, Schuring and Smith in 2006 \cite{PSS} used a transformation optics approach, using invariance properties of the governing Maxwell's equations to design invisibility cloaks at microwave frequencies. Leonhardt in 2006 \cite{Leon} discusses an optical conformal mapping based cloaking scheme. For the ideal/perfect invisibility cloaking considered in \cite{PSS, Leon}, it is a singular 'blow-up-a-point' transformation. The cloaking media achieved in this way inevitably have singular materials parameters and require design of {\it metamaterials}. The singularity poses much challenge to both theoretical analysis and practical construction.
While several proof-of-concept prototypes have been proposed as cloaks, several challenges still remain in developing fully functional devices capable of fully cloaking objects. A lot of current academic and industrial research in material science is focused on development of such metamaterials from proof-of-concept prototypes to practical devices. See \cite{ ALU} for more details on this topic.

In order to avoid the singular structure, it is natural to introduce regularizations into the construction, and instead of the perfect cloak, one considers the approximate cloak or near-cloak. In order to handle the singular structure from the perfect cloaking constructions, the papers \cite{GKLU, GKLU10, Ruan} used this truncation of singularities methods to approach the nearly
cloaking theory, whereas other papers regularize the 'blow-up-a-point' transformation to 'blow-up-a-small-region' transformation. The small-inclusion-blowup method was studied in  \cite{Ammari, KSVW} for the conductivity model.

The papers \cite{GLU,GLU1} considered the case of electrostatics, which is optics at frequency zero. These papers provide counter examples to uniqueness in Calder\'on Problem, which is the inverse problem
for electrostatics which lies at the heart of Electrical Impedance Tomography [EIT]. EIT consists of determining the electrical conductivity of a medium filling
a region $\Omega$ by making voltage and current measurements at the boundary
$\partial \Omega$ and was first proposed in \cite{Calderon}. The fundamental mathematical idea behind cloaking is using the invariance of a coordinate transformation for specific systems, such as conductivity, acoustic, electromagnetic, and elasticity systems. We refer readers to the article \cite{Gunther} for a nice overview of development in EIT and cloaking for electrostatics. We also refer the readers to \cite{GLU2, GLU3, Hongyu, LZ, LHUG, GKLU, Lin} for the theory behind cloaking in various systems and related developments.

In this paper we will focus our attention for cloaking in electrostatics and  consider the following divergence type quasi-linear elliptic boundary value problem.
\begin{equation}\label{basic}
\begin{aligned}
-div( A(x,u) \nabla u )&= 0 \mbox{ in } \Omega  \\
u &= f \mbox{ on } \partial \Omega
\end{aligned}
\end{equation}
where $\Omega \subset \mathbb{R}^N$, $N \geq 2$,  is a bounded open set with smooth enough boundary and $A(x,t)$ is a non-negative symmetric matrix valued function in $\Omega \times \mathbb{R}$ which satisfies certain structure conditions. \eqref{basic} is a generalization of the conductivity equation considered in Electrical Impedance Tomography.

Let us now introduce the basic mathematical set up.
Consider $\Omega \subset \mathbb{R}^N$, $N \geq 2$, a bounded open set with Lipschitz boundary. Let $\mathcal{M}(\alpha, \beta,L;\Omega \times \mathbb{R})$ with
$0<\alpha<\beta < \infty$ and $L>0$ denote the set of all real $N\times N$ symmetric matrices $A(x,t)$ of functions defined almost
everywhere on $\Omega \times \mathbb{R}$ such that if $A(x,t)=[a_{kl}(x,t)]_{1\leq k,l\leq N}$ then 
\begin{enumerate}
 \item $a_{kl}(x,t)=a_{lk}(x,t)\ \forall l, k=1,..,N $
 \item $ (A(x,t)\xi,\xi)\geq \alpha|\xi|^2,\ |A(x,t)\xi|\leq \beta|\xi|,\ \ \forall\xi \in \mathbb{R}^N,\ \mbox{ a.e. }x\in\Omega \mbox{ and, }$
\item $|a_{kl}(x,t) - a_{kl}(x,s)| \leq L|t-s|\mbox{ for  a.e } x \in \Omega \mbox{ and any } t, s \mbox{ in } \mathbb{R}.$
\end{enumerate}

Under the above conditions, we show in Theorem \ref{appendix} that the following boundary value problem has a unique solution $u \in H^1(\Omega)$,
\begin{equation}\label{basiceqn}
\begin{aligned}
-div (A(x,u) \nabla u ) &= 0 \mbox{ on } \Omega\\
 u &= f \in H^{1/2}(\partial \Omega)
\end{aligned}
\end{equation}

The Dirichlet-to-Neumann map for the boundary value problem \eqref{basiceqn} is defined formally as 
the map $f: \to \Lambda_{A}f$
\begin{align*}
\Lambda_{A}f = \nu \cdot A(x,f)\nabla u|_{\partial \Omega}
\end{align*}
where $\nu$ is the outer unit normal to $\partial \Omega$. It is shown in Appendix A  that one can define the Dirichlet-to-Neumann (DN) map for the equation \eqref{basiceqn} in a weak sense as follows.
\begin{align*}
\Lambda_{A}: H^{\frac{1}{2}}(\partial \Omega) \to H^{-\frac{1}{2}}(\partial \Omega)
\end{align*}
 The inverse problem is to recover the quasi-linear co-efficient matrix $A(x,t)$, also called the conductivity from the knowledge of $\Lambda_{A}$. Before we provide the definition of cloaking, since the problem of cloaking is essentially that of non-uniqueness, we digress a bit and mention some previous work regarding uniqueness in the inverse problem for the equation considered in \eqref{basiceqn}.
 
 In the isotropic case, that is, $A(x,t) = a(x, t)I$ where $I$ denotes the identity matrix and $a$ is a positive $C^{2,\gamma}(\overline{\Omega} \times \mathbb{R})$ function having a uniform positive lower bound on $\overline{\Omega} \times [ - s, s]$ for each $s > 0$, the Dirichlet
to Neumann map $\Lambda_{a}$ determines uniquely the scalar coefficient $a(x,t)$ on $\overline{ \Omega }\times \mathbb{R}$. This uniqueness result was first proved for the linear case (i.e when $a$ is a function of $x$ alone) in the fundamental paper \cite{SU} for $N \geq 3$ for and in \cite{Nachman} for $N = 2$;  and in \cite{Sun} for the quasilinear case. Subsequent work on  unique identification of less regular $a$ in the isotropic linear case has been done in \cite{HabermanTataru,Haberman,CaroRogers} among others for dimensions $3$ and higher and in \cite{BrownUhlmann,Paivarinta} for dimension $2$. 

For the anisotropic/ matrix valued case, it is well known that one cannot recover the coefficient $A(x,t)$ itself because of the following invariance property for the DN map. Choose a smooth diffeomorphism  $\Phi:\Omega \to  \Omega$ such that 
$\Phi = Id$ on $\partial \Omega$ and define
\begin{align}\label{pushforward}
\Phi_{*}A(x,t) = \frac{D\Phi(x)^{T}A(x,t)D\Phi(x)}{|D\Phi|} \circ \Phi^{-1}(x)
\end{align}
We make the change of variables $y = \Phi(x)$ in \eqref{basiceqn} to get
\begin{equation*}
\int \limits_{\Omega} \,\sum \limits_{i,j=1}^{n} a_{ij}(x,u)\frac{\partial u}{\partial x^{i}}\frac{\partial u}{\partial x^{j}}\,dx = \int \limits_{\Omega} \sum \limits_{i,j=1}^{n}  a_{ij}(x,u) \frac{\partial u}{\partial y^{k}} \frac{\partial y^{k}}{\partial x^{i}} \frac{\partial u}{\partial y^{l}}\frac{\partial y^{l}}{\partial x^{j}} det\left(\frac{\partial x}{\partial y}\right)\,dy.
\end{equation*}
We can write this more compactly as 
\begin{equation*}
\int \limits_{\Omega} \left<A(x,u) \nabla_{x}u,\nabla_{x}u\right>\,dx = \int \limits_{\Omega} \left<\Phi_{*}A(y,u)\nabla_{y}u, \nabla_{y}u \right>\,dy.
\end{equation*}
where $\Phi_{*}A$ is as defined in \eqref{pushforward}.

Since $\Phi$ is identity at $\partial \Omega$, the change of variables does not affect the Dirichlet data and we obtain 
\begin{equation}\label{diff}
\Lambda_{A} = \Lambda_{\Phi_{*}A}
\end{equation}
Thus, for matrix valued co-efficients $A(x,t)$, one can expect uniqueness only modulo such a diffeomorphism. For dimension 2, in the linear case, such uniqueness up to difeomorphsim has been proved in \cite{Sunanisotropic,Astala} and for dimension 3 and higher in \cite{LeeUhlmann}. For the quasilinear case, Sun and Uhlmann in \cite{SunUhlmann} showed uniqueness up to diffeomorphism in dimension $2$ assuming $C^{2,\gamma}$, $0<\gamma<1$ smoothness of the coefficients $A$. They also proved uniqueness up to diffeomorphism for $N \geq 3$ for real analytic coefficients $A(x,t)$. Whether uniqueness up to diffeomorphism can be shown for less regular anisotropic quasi-linear coefficient $A(x,t)$ is an interesting question and remains open. 

Equations of the form \eqref{basic} are important and arise in many applications (eg, the stationary form of Richards equation \cite{BY}, the modeling of thermal conductivity of the Earth's crust \cite{WHN} or heat conduction in composite materials \cite{KL}). One of the goals of our work is to extend the result obtained in \cite{GLU,GLU1,KSVW} to the quasi-linear elliptic equation \eqref{basiceqn}. We propose a change of variable scheme, similar to the one in  \cite{GLU,GLU1}  and show how one can, in principle, obtain perfect cloaking, in the context of the equation considered in \eqref{basiceqn}, using singular change of variables and approximate cloaking using a regular change of variables. For approximate cloaking we use the small inclusion blow up method as in \cite{KSVW, Ammari}. The singularity and extreme anisotropy resulting from a singular change of variables pose a great challenge in manufacturing invisibility devices.  The construction of approximate cloaks using  regular change of variables is more tractable. However, these approximate cloaks, though non-singular are still anisotropic. The other major goal of our paper is to construct approximate { \it isotropic} cloaks. This will be accomplished using techniques of homogenization. First we construct approximate anisotropic cloaks using regular change of variables. Next, within the framework of homogenization, we approximate each approximate regular anisotropic cloak by a sequence of regular isotropic cloaks. Homogenization process for constructing isotropic regular approximate cloaks has been considered for a linear equation in \cite{GYLU}. We wish to extend the construction in \cite{GYLU} to the quasi-linear equation considered in \eqref{basic}. To the best of our knowledge, construction of approximate cloaks within the framework of homogenization for a quasi-linear elliptic equation has been done for the first time here. 
\subsection{Definition of Cloaking} 
We now provide a mathematical definition of cloaking for the quasi-linear elliptic partial differential equation considered in \eqref{basiceqn}. 
\begin{definition}\label{definitioncloaking}
Let $E \subset \Omega$ be fixed and let $\sigma_c: {\Omega \setminus E }\times \mathbb{R}$ be a non negative
matrix valued function defined on ${\Omega \setminus E} \times \mathbb{R}$. We say $\sigma_c$ cloaks $E$ if its any extension across $E$ of of following form
\[\sigma_A(x,t)= \begin{cases} 
      A(x,t) & (x,t) \in E \times \mathbb{R}\\
      \sigma_c(x,t) & (x,t) \in {\Omega \setminus E}\times \mathbb{R} \\
   \end{cases}
\]
produces the same Dirichlet-to-Neumann map as a uniform isotropic region irrespective of the choice of $A(x,t) \in \mathcal{M}(\alpha,\beta,L;\Omega \times \mathbb{R})$.
\end{definition}

That is, $\sigma_c$ cloaks $E$ in the sense of Definition 1 if $\Lambda_{\sigma_A} = \Lambda_{1}$ regardless of the choice of the extension $A(x,t)$. 

Suppose $\sigma_c(x,t)$ cloaks $E $ in the sense of Definition 1 and let $\Omega'$ be any domain containing $\Omega$. Then the Dirichlet-to-Neumann map for
\[\sigma_{\tilde A(x,t)}= \begin{cases} 
      A(x,t) & (x,t) \in E \times \mathbb{R}\\
      \sigma_c(x,t) & (x,t) \in \Omega \setminus E \times \mathbb{R} \\
      I & (x,t) \in \Omega'\setminus \Omega \times \mathbb{R}
   \end{cases}
\]
is independent of $A$ and is equal to $\Lambda_{I}$. This extension argument produces many other examples. Indeed, if $\sigma_c$ cloaks $E$ in the sense of Definition 1, then the extension of $\sigma_{c}$ by $I$ outside $\Omega$ cloaks $E$ in any larger domain $\Omega'$ containing $\Omega$. If cloaking is possible,  measurements made on the boundary i.e the knowledge of Dirichlet-to-Neumann map is not enough to detect the presence of an arbitrary inclusion inside the cloaked region.  Since $A(x,t)$ can be arbitrary, an equality of the form \eqref{diff} cannot possibly hold and thus cloaking is essentially a non-uniqueness result. 

This paper is organized as follows. In section 2, we introduce a regular change of variables scheme which will give us the desired approximate cloaking. In section 3, we introduce a singular change of variables and show how perfect cloaking can be achieved. The analysis in this section is essentially a simple extension of the arguments in \cite{GLU,GLU1,KSVW}. Section 4 is devoted to using homogenization techniques for constructing regular isotropic cloaks. We begin this section by recalling the basic notions of H-convergence in the linear case. Following that we perform the periodic homogenization in the quasi-linear settings. This enables us to construct regular isotropic cloak in the sense made precise in Section 4. In Appendix A, we prove existence and uniqueness for the boundary value problem \eqref{basiceqn} and show that is possible to define, in a weak sense, the DN map associated with \eqref{basiceqn}. Moreover, we also state a result on higher regularity for the solutions to \eqref{basic} which will be used in Section 4. Henceforth, we consider the physical dimensions $N=2,3$.

\section{Regular Change of Variables}
In this section, we apply a regular change of variables and nearly cloak $E$ in the sense made precise below.
For simplicity, we let $\Omega = B_2$ and restrict our attention the case when $B_1 =E$ needs to be nearly cloaked.  We extend the result to non-radial domains later.

The basic premise is as follows. Consider a small ball of radius $r$, $B_r$ centered at $0$ where $r <1$. Construct a map $F^r(x):  B_2 \to  B_2$ with the following properties. 
\begin{enumerate}
\item[1)] $F^r$ is continuous and piecewise smooth. 
\item[2)] $F^r$ expands $B_r$ to $B_1$ and maps $B_2$ to itself.
\item [3)] $F^r = x$ on $\partial B_2$.
\end{enumerate}
It is easy to see that the following candidate for $F^r$ satisfies the above properties.
\begin{equation*}F^r(x)= \begin{cases} 
      \frac{x}{r} & |x| \leq r \\
      (\frac{2-2r}{2-r} + \frac{1}{2-r}|x|)\frac{x}{|x|}& r \leq |x| \leq 2   \end{cases}
\end{equation*}
Note that $F^r$ is continuous, piecewise smooth and  non-singular and $(F^r)^{-1}$ is also continuous and piecewise smooth. Consider
\begin{equation}\label{sigmaar}
\sigma_A^{r}(x,t)= \begin{cases} 
      A(x,t) & (x,t) \in B_{1} \times \mathbb{R}\\
      F^{r}_{*}1 & (x,t) \in B_2 \setminus B_1 \times \mathbb{R} \\
   \end{cases}
\end{equation}
where $A(x,t) \in \mathcal{M}(\alpha,\beta,L;B_2 \times \mathbb{R})$. By nearly cloaking, we mean that the following must hold
\begin{equation}\label{nearlycloaking}
|\langle\Lambda_{\sigma_A^{r}}f,g \rangle - \langle \Lambda_{1}f, g \rangle |= \mathrm{o}(1) ||f||_{H^{\frac{1}{2}}(\partial B_2)}||g||_{H^{\frac{1}{2}}(\partial B_2)} \mbox{ for any }f,g \in H^{\frac{1}{2}}(\partial B_2)
\end{equation}
where the $\mathrm{o}(1)$ term is independent of $f$ and $g$. 
%In other words, $ F^{r}_{*}1$ {\it approximately/nearly} cloaks $B_1$.

\eqref{nearlycloaking} is equivalent to
\begin{align}\label{new}
||\Lambda_{\sigma_A^{r}} - \Lambda_1||_{H^{\frac{1}{2}}(\partial B_2) \times H^{-\frac{1}{2}}(\partial B_2)} &= \sup_{f, g \in H^{\frac{1}{2}}(\partial B_2) }|\langle\Lambda_{\sigma_A^{r}}f,g \rangle - \langle \Lambda_{1}f, g \rangle |\notag \\&=\mathrm{o}(1) ||f||_{H^{\frac{1}{2}}(\partial B_2)}||g||_{H^{\frac{1}{2}}(\partial B_2)} 
\end{align} 
By \eqref{diff}, the DN map for $\sigma^r_A$ is identical to that of $(F^r)^{-1}_{*} \sigma^r_A$.
We will show 
\[|\langle \Lambda_{(F^{r})^{-1}_{*} \sigma_A^r}f, g \rangle  - \langle \Lambda_{1}f, g \rangle | =\mathrm{o}(1) ||f||_{H^{\frac{1}{2}}(\partial B_2)}||g||_{H^{\frac{1}{2}}(\partial B_2)} \mbox{ for any }f,g \in H^{\frac{1}{2}}(\partial B_2)
\] where the $\mathrm{o}(1)$ term is independent of $f$ and $g$ and where
 \begin{equation*}(F^{r})^{-1}_{*} \sigma_A^r = \begin{cases} 
      (F^{r})^{-1}_{*} A =\widetilde{A}^r(x,t)& (x,t) \in B_{r} \times \mathbb{R}\\
     1 &(x,t)\in (B_2 \setminus B_r) \times \mathbb{R} 
   \end{cases}
\end{equation*}

Let us now explicitly calculate $\widetilde{A}^r(x,t)$. We note that for $\Phi(x) = (F^r)^{-1}(x)$, $D \Phi(x) = r I$ for $x \in B_1$. This implies that $|D\Phi| = r^N$. 
From \eqref{pushforward}, it follows that\\
$\widetilde A^r(x,t) =(F^{r})^{-1}_{*} A= \frac{1}{r^{N-2}} A(\frac{x}{r},t)$ for $x \in B_r$. 
Hence
if $A \in \mathcal{M}(\alpha,\beta,L;B_1 \times \mathbb{R})$ then\\
 $\widetilde A^r(x,t) \in \mathcal{M}(\frac{\alpha}{r^{N-2}},\frac{\beta}{ r^{N-2}}, L;B_r \times \mathbb{R})$. This ultimately implies that \\for $r\ll 1$, ${F^{r}_{*}}^{-1} \sigma_A^r \in \mathcal{M}(1,\frac{\beta}{ r^{N-2}}, L;B_r \times \mathbb{R})$.

Let us now fix $f \in H^{1/2}(\partial B_2)$. Let $u^{r,f}\in H^1(B_2)$ uniquely solve
\begin{equation}\label{urho}
\begin{aligned}
-div ((F^{r})^{-1}_{*}\sigma_{A}(x,u^{r,f})\nabla u^{r,f}) &= 0 \mbox{ in } B_2  \\
u^{r,f} &= f \mbox{ on } \partial B_2
\end{aligned}
\end{equation}
(Unique solution to \eqref{urho} is indeed guaranteed by the  existence and uniqueness result proved in Theorem \ref{appendix}).

Let $v^{f}\in H^1(B_2)$ solve 
\begin{equation}\label{vf}
\begin{aligned}
-\Delta v^{f} &= 0 \mbox{ in } B_2  \\
v^{f} &= f \mbox{ on } \partial B_2
\end{aligned}
\end{equation}
Note that $u^{r,f} - v^{f} \in H^1_{0}(B_2)$. Using coercivity for $(F^{r})^{-1}_{*}\sigma_{A}$ gives us 
\begin{equation}\label{sb5}
||\nabla(u^{r,f} - v^{f})||^2_{L^{2}(B_2)} \leq  | \int_{B_2} (F^{r})^{-1}_{*} \sigma_A(x,u^{r,f}) \nabla(u^{r,f} - v^f) \cdot \nabla(u^{r,f} - v^f)\ dx|.
\end{equation}
Now
\begin{align}\label{computation}
\mbox{ r.h.s. of \eqref{sb5} }
&=| \int_{B_2} (F^{r})^{-1}_{*} \sigma_A(x,u^{r,f}) \nabla(u^{r,f}) \cdot \nabla(u^{r,f} - v^f)dx\notag\\
&\qquad-\int_{B_2} (F^{r})^{-1}_{*}\sigma_A(x,u^{r,f}) \nabla(v^{f}) \cdot \nabla(u^{r,f} - v^f)\, dx| \notag \\
&=|-\int_{B_2}(F^{r})^{-1}_{*}\sigma_A(x,u^{r,f}) \nabla v^{f} \cdot \nabla(u^{r,f} - v^f)\,dx|\;(\mbox{as }u^{r,f} \mbox{ solves \eqref{urho}}) \notag \\
%%&=|-\int_{B_2}(F^{r})^{-1}_{*}(\sigma_A(x,u^{r,f}) -I) \nabla v^{f} \cdot \nabla(u^{r,f} - v^f)dx \notag \\
%%&\qquad- \int_{B_2} \nabla v^{f} \cdot \nabla(u^{r,f} - v^f)\,dx| \notag \\
&=|-\int_{B_2}((F^{r})^{-1}_{*}\sigma_A(x,u^{r,f}) -I) \nabla v^{f} \cdot \nabla(u^{r,f} - v^f)\,dx| \mbox{ (as } v^{f} \mbox{ solves \eqref{vf}) } \notag \\
&=|\int_{B_r}  ((F^{r})^{-1}_{*}\sigma_A(x,u^{r,f}) -I) \nabla v^{f} \cdot \nabla(u^{r,f} - v^f)\,dx |) \notag \\
&=|\int_{B_r} (\widetilde{A}^r(x,u^{r,f}) -I) \nabla v^{f} \cdot \nabla(u^{r,f} - v^f) \,dx|
\end{align} 
We note that $||\widetilde{A}^r(x,t)||_{L^{\infty}(B_r)} \leq \frac{C}{r^{N-2}}$ where the constant $C$ is independent of $r$.  We apply H\"older's inequality to the last line in \eqref{computation} to obtain
\[||\nabla(u^{r,f} - v^{f})||^2_{L^{2}(B_2)} \leq C r^{\frac{N}{p_1} - N +2} ||\nabla v^f||_{L^{p_2}(B_{r})} ||\nabla (u^{r,f} - v^f)||_{L^{2}(B_r)}
\]
where $C$ is independent of $r$ and $\frac{1}{p_1} + \frac{1}{p_2}=\frac{1}{2}$. We thus have
\begin{align*}
||\nabla(u^{r,f} - v^{f})||^2_{L^{2}(B_2)} &\leq C r^{\frac{N}{p_1} - N +2} ||\nabla v^f||_{L^{p_2}(B_{r})} ||\nabla (u^{r,f} - v^f)||_{L^{2}(B_r)} \notag \\
& \leq C  r^{\frac{N}{p_1} - N +2} ||\nabla v^f||_{L^{p_2}(B_{r})} ||\nabla (u^{r,f} - v^f)||_{L^{2}(B_2)}
\end{align*}
By Poincare's inequality we can say that
\begin{equation}\label{xkcd}
||u^{r,f} - v^f||_{H^1(B_2)} \leq C r^{\frac{N}{p_1} - N +2} ||\nabla v^f||_{L^{p_2}(B_r)}
\end{equation}
By Corollary 6.3 in \cite{GT},  we can say that
\begin{equation}\label{GT1}
||\nabla v^f||_{L^{\infty}(B_r)} \leq  C||v^f||_{L^{\infty}(B_1)}
\end{equation}
where $C$ is independent of $r$ and $f$.

We now use  \cite[Theorem 8.24]{GT} to conclude that
\begin{equation}\label{GT2}
||v^f||_{L^{\infty}(B_1)} \leq C ||v^f||_{L^2(B_2)}
\end{equation}
where $C$ is independent of $r$ and $f$.  \eqref{GT1} and \eqref{GT2} together imply
\begin{align}\label{GT3}
||\nabla v^f||_{L^{p_2}(B_r)} &\leq C \,r^{\frac{N}{p_2}}||v^f||_{L^2(B_2)} \leq C\, r^{\frac{N}{p_2}}||f||_{H^{\frac{1}{2}}(\partial B_2)}
\end{align}
\eqref{xkcd} and \eqref{GT3} hence give us 
\begin{align}\label{key}
||u^{r,f} - v^f||_{H^1(B_2)} &\leq C r^{\frac{N}{p_1} + \frac{N}{p_2}-N+2} ||f||_{H^{\frac{1}{2}}(\partial B_2)} \notag \\
&= C r^{-\frac{N}{2}+2} ||f||_{H^{\frac{1}{2}}(\partial B_2)}
\end{align}
where $C$ is independent of $r$ and $f$ also.

Let $g \in H^{\frac{1}{2}}(\partial B_2)$ be arbitrary. We know that there exists a  unique $v^g\in H^1(B_2\setminus B_1)$ which solves the following boundary value problem
	\begin{equation}\label{newstuff}
	\begin{aligned}
	-\Delta v^g &= 0 \mbox{ in }B_2 \setminus B_1 \\
	v^g&=0 \mbox{ on } \partial B_1 \\
	v^g&=g \mbox{ on } \partial B_2
	\end{aligned}
	\end{equation}
such that 
\begin{equation}\label{key1}
||v^g||_{H^{1}(B_2 \setminus B_1)} \leq C ||g||_{H^{\frac{1}{2}}(\partial B_2)}
\end{equation}
	where $C$ is independent of $g$. 
		
From \eqref{urho}, \eqref{vf}, \eqref{newstuff}, \eqref{key}, \eqref{key1} and the definition of $(F^{r})^{-1}_{*} \sigma_A^r$	, we see that 
\begin{align}\label{yes1}
|\langle \Lambda_{(F^{r})^{-1}_{*} \sigma_A^r}f, g \rangle  - \langle \Lambda_{1}f, g \rangle| &=|\int_{\partial B_2} \frac{\partial u^{r,f}}{\partial \nu}v^g\,dS - \int_{\partial B_2} \frac{\partial v^{f}}{\partial \nu}v^g\,dS|\notag \\
 &=|\int_{B_2 \setminus B_1} (\nabla u^{r,f} \cdot \nabla v^{g}  - \nabla v^f \cdot \nabla v^g)\,dx| \notag \\
&\leq \int_{B_2 \setminus B_1} ||\nabla u^{r,f} - \nabla v^f||_{L^{2}(B_2 \setminus B_1)} ||v^g||_{H^1(B_2 \setminus B_1)} \notag \\
&\leq C r^{-\frac{N}{2}+2} ||f||_{H^{\frac{1}{2}}(\partial B_2)}||g||_{H^{\frac{1}{2}}(\partial B_2)} 
\end{align}
where $C$ is independent of $r$, $f$ and $g$.
	
\eqref{yes1} implies that
for $N=2,3$, we obtain \eqref{new}. 

\subsection{Faster decay}

In this subsection we derive an improved rate of convergence in \eqref{yes1} at the cost of choosing smoother boundary data. Fix $f \in H^{\frac{3}{2}}(\partial B_2)$.

Since $u^{r,f}$ solves \eqref{urho} and $v^f$ solves \eqref{vf}, we get, for any $w \in H^1_{0}(B_2)$
\begin{align}\label{rateincrease1}
\int_{B_{2}} \nabla(u^{r,f} - v^f) \cdot \nabla w\,dx = \int_{B_r} (I - \widetilde A^r(x,u^{r,f})) \nabla u^{r,f} \cdot \nabla w\,dx
\end{align}
Choose $w$ which uniquely solves 
\begin{equation}\label{rateincrease2}
\begin{aligned}
-\Delta w &= u^{r,f} - v^f \mbox{ in } B_2  \\
w &= 0 \mbox{ on } \partial B_2
\end{aligned}
\end{equation}
Since $u^{r,f} - v^f \in L^2(B_2)$, $w \in H^2(B_2)$ with
\begin{equation}\label{h2r}
||w||_{H^2(B_2)}\leq C||u^{r,f}-v^f||_{L^2(B_2)}.
\end{equation}
where $C$ is independent of $r$. 

By Sobolev embedding, this implies $\nabla w \in L^q(B_2)$ for any $1 \leq q<\infty$ for $N=2$ and $\nabla w \in L^q(B_2)$ for any $1 \leq q \leq \frac{2N}{N-2}$ for $N = 3$.
\begin{align*}
\int_{B_2} (u^{r,f} - v^f)^2\,dx &= \int_{B_r} (I - \widetilde {A}^r(x,u^{r,f})) \nabla u^{r,f}\cdot  \nabla w\,dx\notag\\
& \leq \frac{C}{r^{N-2}} ||\nabla u^{r,f}||_{L^{p^\prime}(B_r)} ||\nabla w ||_{L^{p}(B_r)}\quad\mbox{where $\frac{1}{p}+\frac{1}{p^\prime}=1,\ $} p>2\notag\\
 &\leq Cr^{\frac{N}{q^{\prime}} - N+2} ||\nabla u^{r,f}||_{L^{p^\prime}(B_r)}  ||\nabla w||_{L^q(B_2)}\quad \mbox{where $\frac{1}{q} + \frac{1}{q^{\prime}} = \frac{1}{p}$} \end{align*}
 Note that we have the following conditions
 \begin{equation}\label{exponents}
 \begin{aligned}
 &2<p<q<\infty \quad \mbox{ for  } N= 2 \\
 &2<p<q\leq \frac{2N}{N-2} \mbox{ for } N =3
 \end{aligned}
 \end{equation}
Then by using \eqref{h2r} we have 
\begin{align*}
||u^{r,f} -v^f||_{L^2(B_2)}&\leq Cr^{\frac{N}{q^{\prime}} - N+2} ||\nabla u^{r,f}||_{L^{p^\prime}(B_r)} \notag \\
&\leq Cr^{\frac{N}{q^{\prime}} - N+2} \left[||\nabla(u^{r,f}- v^f)||_{L^{p^\prime}(B_r)} + ||\nabla v^f||_{L^{p^\prime}(B_r)} \right] \notag \\
&\leq Cr^{\frac{N}{q^{\prime}} - N+2} \left[||\nabla(u^{r,f}- v^f)||_{L^{2}(B_r)} || \, \mathbbm{1}_{B_r}||_{L^{s}(B_r)} + ||\nabla v^f||_{L^{\infty}(B_r)} || \, \mathbbm{1}||_{L^{p^\prime}(B_r)}  \right]
\end{align*}
where, $\frac{1}{p^\prime}=\frac{1}{2}+\frac{1}{s}$. Now by using
\eqref{GT1}, \eqref{GT2} and \eqref{key} together give us 
\[
||u^{r,f} -v^f||_{L^2(B_2)} \leq Cr^{\frac{N}{q^{\prime}} - N+2} \left(r^{\frac{-N}{2} + 2 +\frac{N}{s}}  ||f||_{H^{\frac{1}{2}}(\partial B_2)} + r^{\frac{N}{p^\prime}} ||f||_{H^{\frac{1}{2}}(\partial B_2)}\right)
\]
or,
\begin{align}\label{fasterdecay}
||u^{r,f} - v^f||_{L^2(B_2)} &\leq C r^{N(1-\frac{1}{p})} ||f||_{H^{\frac{1}{2}}(\partial B_2)} \notag \\
&= \mathcal{O}(r^{-N - \frac{N}{q} +4})
\end{align}
\eqref{fasterdecay} implies that for $N=3$, we have a decay rate of $\mathcal{O}(r^{\frac{q-3}{q}})$. From \eqref{exponents}, we can optimize this decay rate to $\mathcal{O}(r^{\frac{6-3}{6}}) = \mathcal{O}(r^\frac{1}{2})$. This is of the same order as in \eqref{key} for $N=3$. 
 
However, from \eqref{fasterdecay} and \eqref{exponents}, we can conclude that For $N=2$, the best possible decay is  $\mathcal{O}(r^{2 - \frac{2}{q}})$ for any $q>2$. This decay rate is faster than the one we obtained in \eqref{key}.

Now, by \cite[Theorem 9.13]{GT}, since $u^{r,f} - v^f$ is harmonic in $B_2 \setminus \overline{ B_r}$ and $u^{r,f} - v^f = 0$ on $\partial B_2$ we have 
\begin{equation}\label{H2estimate}
||u^{r,f} - v^f||_{H^{2}(\Omega^\prime)} \leq C ||u^{r,f} - v^{f}||_{L^{2}(B_2 \setminus \overline{ B_r)}}
\end{equation}
where $\Omega^\prime = B_2 \setminus \overline{B_{r + \delta}}$ where $\delta >0$ is small enough. \eqref{fasterdecay} and \eqref{H2estimate} along with the trace theorem for Sobolev spaces imply
\begin{align}\label{218}
||\frac{\partial u^{r,f}}{\partial \nu} - \frac{\partial v^{f}}{\partial \nu}||_{H^{1/2}(\partial B_2)} =\mathrm{o}(1) ||f||_{H^{3/2}(\partial B_2)}.
\end{align}

The decay estimates for nearly cloaking here are weaker than the one in \cite{KSVW} where a decay of $\mathcal{O}(r^N)$ is obtained. In our case, though we have a slower decay rate, it is sufficient to show approximate cloaking. 

So far, we focused on the radial setting because of its explicit character. A similar argument as to the one provided in this section in fact proves
\begin{corollary}
	Let $H: B_2 \to \Omega$ be a Lipschitz continuous map with Lipschitz continuous inverse and let $E = H(B_1)$. Then $G^r = H \circ F^r \circ H^{-1}: \Omega \to \Omega$ is piecewise smooth with the following properties.
	\begin{enumerate}
		\item
		$G^r$ expands $B_r$ to $E$,
		\item
		$G^r(x) = x$ on $\partial \Omega$.
		\end{enumerate}
	Let
\begin{align}\label{sigmaar1}
\sigma_A^{r}(x,t)= \begin{cases} 
A(x,t) & (x,t) \in E \times \mathbb{R}\\
G^{r}_{*}1 & (x,t) \in \Omega \setminus E \times \mathbb{R} \\
\end{cases}
\end{align}
where $A(x,t) \in \mathcal{M}(\alpha,\beta,L;\
\Omega \times \mathbb{R})$. Then the following holds
\begin{equation}\label{nearlycloaking1}
|\langle\Lambda_{\sigma_A^{r}}f,g \rangle - \langle \Lambda_{1}f, g \rangle |= \mathrm{o}(1) ||f||_{H^{\frac{1}{2}}(\partial \Omega)}||g||_{H^{\frac{1}{2}}(\Omega)} \mbox{ for any }f,g \in H^{\frac{1}{2}}(\partial \Omega)
\end{equation}
In other words, $ G^{r}_{*}1$ {\it approximately/nearly} cloaks $E$.
	\end{corollary}
%Note to self: Note that boundary of $\partial \Omega'$ has two layers. Trace theorem will imply that sum over each of the two boundaries decays like $\rho^{n/2 +1}$ which in particular implies that the term on \partial B_2 has the decay which is what we want

\section{Perfect cloaking}
We now show how $E \subset \Omega$ can be perfectly cloaked using a singular change of variables. For simplicity we first take $E=B_1$ and $\Omega=B_2$. The analysis in this section mirrors that in Section 4 of \cite{KSVW}. 

Let us define 
\begin{equation}\label{definitionoff}
F(x) = \left(1 + \frac{|x|}{2} \right)\frac{x}{|x|}
\end{equation}
$F$ has the following properties 
\begin{enumerate}
\item[1)] $F$ is smooth except at $x=0$.
\item[2)] $F$ expands $0$ to $B_1$ and maps $B_2$ to itself.
\item [3)] $F = x$ on $\partial B_2$.
\end{enumerate}
Our candidate for perfect cloaking will be $F_{*}1$. We first calculate $F_{*}1$ explicitly.
Note that 
\begin{equation}\label{DF}
DF = \left ( \frac{1}{2} + \frac{1}{|x|} \right )I - \frac{1}{|x|} \hat x \hat x^{T}
\end{equation}
for $|x| >0$, where $I$ is the identity matrix and $\hat x = \frac{x}{|x|}$ and  $DF$ is a symmetric matrix such that
\begin{enumerate}
\item[a)] $\hat x$ is an eigenvector with eigen-value $1/2$
\item[b)] $\hat x^{\bot}$ is a (N-1) dimensional eigen-space with eigen-value $\frac{1}{2} + \frac{1}{|x|}$
\end{enumerate}
The determinant is thus 
\begin{equation}\label{det}
det(DF) = \frac{1}{2}\left(\frac{1}{2} + \frac{1}{|x|} \right )^{N-1}
\end{equation}
Hence, whenever $1 < |y| <2$, we have 
\begin{equation}\label{evalues}
F_{*}1(y) = \frac{2^N}{(2 + |x|)^{N-1}}\left[\left(\frac{|x|^{N-1}}{4} + |x|^{N-2} + |x|^{N-3} \right)(I - \hat x \hat x ^{T})  + \frac{1}{4}|x|^{N-1} \hat x \hat x^{T}\right]
\end{equation}
where the right hand side is evaluated at $x = F^{-1}(y) = 2(|y| -1)\frac{y}{|y|}$

As $|y| \to 1$, $F_{*}1$ becomes singular. In fact, following the same arguments as in \cite{KSVW}, we get
\begin{enumerate}
\item[*]
When $N=2$, one eigenvalue of $F_{*}1 \to 0$ and the other tends to $\infty$ as $|y| \to 1$
\item[*]
For $N=3$, one eigenvalue goes to $0$ while others remain finite as $|y| \to 1$.
\end{enumerate}
Consider $\sigma_A$ defined as 
\begin{equation}\label{defnsigma}
\sigma_A(y,t)= \begin{cases} 
      A(y,t) & (y,t) \in B_{1} \times \mathbb{R}\\
      F_{*}1 & (y,t) \in B_2 \setminus B_1 \times \mathbb{R} \\
   \end{cases}
\end{equation}
where $A(y,t)\in \mathcal{M}(\alpha,\beta,L;B_1 \times \mathbb{R})$.

Define $v$ by
\begin{equation}\label{defnv}
v(y)=
\begin{cases}
u(0) & y \in B_1  \\
u(x) & y \in B_2 \setminus B_1 
\end{cases}
\end{equation}
where $u \in H^1(B_2) $ is the harmonic function in $B_2$ with boundary data $f \in H^{1/2}(\partial B_2)$ and $x = F^{-1}(y)$.

We will show that $v$ solves 
\begin{equation}\label{cloakingeqn}
\begin{aligned}
-div_y( \sigma_A(y,v(y)) \nabla_{y} v(y) )&= 0 \mbox{ in } B_2  \\
 v&=f \text{ on } \partial B_2
\end{aligned}
\end{equation}
Since $F_{*}1$ is degenerate near $|y|=1$, it is not immediately clear if \eqref{cloakingeqn} has a  weak solution. 

We need to put some constraints to get a unique weak solution for \eqref{cloakingeqn} since $\sigma_A(y,t)$  is not uniformly elliptic.  As $F_{*}1$ is smooth for $|y|>1$, by elliptic regularity $v$ will be uniformly bounded in any compact subset of $B_2 \setminus \overline{B_1}$. %(Note that, $v$ solves a {\it linear} PDE in $B_2 \setminus \bar B_1$.)
 Since $F_{*}1$ becomes degenerate near $|y|=1$, we ask that any solution $v(y)$ not diverge as $|y| \to 1$. That is, we ask that
\begin{equation}\label{boundedness}
|v(y)| \leq C \mbox{ for } |y| \leq \rho
\end{equation}
for some finite $C$ and $1 < \rho < 2$. 

We first prove a lemma which identifies the value of any solution $v$ to \eqref{cloakingeqn} on $B_2 \setminus \overline{B_1}$ .
\begin{lemma}\label{lemma3.1}
Any solution $v$ to \eqref{cloakingeqn} satisfying \eqref{boundedness} is such that 
\begin{equation*}
v(y) = u(x)\mbox{ for } 1<|y|<2
\end{equation*}
where $x = F^{-1}(y)$ and $u$ us the harmonic function on $B_2$ with the same Dirichlet data as $v$. 
\end{lemma}
\begin{proof}
For any compactly supported test function $\phi$ in $B_2 \setminus \overline{B_1}$, by change of variables, we have
\[
0 = \int_{B_2 \setminus \overline{B_1}}\sigma_A(y,v)\nabla_{y}v \cdot \nabla_{y}\phi\,dy = \int_{B_2 \setminus \{0\}} \nabla_{x}v(F(x)) \cdot \nabla _{x}\phi(F(x))\,dx
\]
We thus see that $v(F(x))$ is weakly harmonic in the punctured ball $B_{2} \setminus \{0\}$. Elliptic regularity implies $v(F(x))$ is strongly harmonic in the punctured ball.

Since, we demand that $v$ satisfy \eqref{boundedness}, $u(x) = v(F(x))$ has a removable singularity at $0$. Thus $u(0)$ is determined by continuity and the extended $u$ is harmonic in the entire ball $B_2$.

Clearly, since $F=x$ on $\partial B_2$, $u$ has the same Dirichlet data as $v$ on $\partial B_2$. The conclusion of the lemma thus holds. 
\end{proof}

We now show that $v$ as defined in \eqref{defnv} solves \eqref{cloakingeqn}.
\begin{lemma}\label{existence}
$v$ as defined in \eqref{defnv} solves
\begin{equation*}
\begin{aligned}
	-div_y( \sigma_A(y,v(y)) \nabla_y v )&= 0 \mbox{ in } B_2 \notag \\
	 v&=f \text{ on } \partial B_2
\end{aligned}
\end{equation*}
where $u$ is the harmonic function on $B_2$ with Dirichlet data $f$. 
\end{lemma}
\begin{proof} 
\mbox{}
\begin{compactenum} 
	\item We first show that $|\nabla v|$ is uniformly bounded  in $ B_r$ for every $r<2$. To see this, note that by chain rule and symmetry of $DF$, we have 
\[\nabla_y v = (DF^{-1})^{T} \nabla_x u = DF^{-1} \nabla_x u
\]
for $ 1<|y|<2$. The matrix $DF^{-1}$ is uniformly bounded by \eqref{DF} and so is $\nabla _x u$, except perhaps near $\partial B_2$ as $u$ is harmonic in whole $B_2$. Hence $|\nabla_y v|$ is bounded on $B_r \setminus B_1$ for any $1\leq r<2$. Moreover $v$ is constant on $B_1$, and continuous across $\partial B_1$. Thus $|\nabla v|$ is uniformly bounded on $B_r$ for $r<2$.  
\item Next we prove that $|\sigma_A(y,v(y)) \nabla_{y} v| $ is also uniformly bounded on $B_r$ for every $r<2$. For $1<|y|<2$, using definition of $\sigma_A$ and chain rule and symmetry of $DF$, we have
\begin{equation}\label{411}
\sigma_A(y,v)\nabla_y v= F_{*}1\, (DF^{-1}) \nabla_x u
\end{equation}
The symmetric matrices $F_{*}1$ and $(DF)^{-1}$ have the same eigenvectors, namely $\hat x$
and $\hat x^{\bot}$. 

Taking $N = 2$, we see that the eigenvalue of $F_{*}1$ in direction $\hat x^{\bot}$ behaves like $|x|^{-1}$, while that of $(DF)^{-1}$ behaves like $|x|$. The eigenvalues of both matrices in direction $\hat x$ are bounded. Thus the product $F_{*}1(DF)^{-1}$ is bounded. This proves Step (2), since $\nabla_x u$ is bounded
away from $\partial B_2$ and $\sigma_A(y,v)\nabla v = 0 $ for $y \in B_1$ as $v$ is defined to be a constant in $B_1$. 

For $N = 3$, this follows directly from the above proved fact that $|\nabla v|$ is uniformly bounded on $B_r$ for every $r<2$. Since $F_{*}1$ is uniformly bounded and by \eqref{DF}, $DF^{-1}$ is uniformly bounded we get that $\sigma_A(y,v(y)) \nabla_{y} v$ is uniformly bounded in $B_r$ for every $r<2$. 
\item
Next we show $\sigma_A(y,v) \nabla_y v \cdot \nu \to 0$ uniformly as $|y| \to 1$ where $\nu$ is the unit outer normal to $\partial B_1$. We have, $\frac{y}{|y|} = \frac{x}{|x|} = \hat x$ and $|y| \to 1 \equiv x \to 0$. We thus need to show that $\hat x $ component of \eqref{411} goes to $0$ as $|x| \to 0$. Since $F_{*}1(DF^{-1})$ is symmetric and $\hat x$ is an eigenvector, it is enough to show that the corresponding eigenvalue tends to $0$. 

We have, from \eqref{det} and \eqref{evalues}, that the eigenvalue corresponding to $\hat x$ is
\[
\frac{2^{N-1}}{(2 + |x|)^{N-1}} |x|^{N-1} \leq |x|^{N-1}
\]
which tends to zero for any $N \geq 2$ as $x \to 0$
\item
We will use the fact that a bounded vector field $X$ is weakly divergence free on $B_2$ iff it is weakly divergence free on $B_2 \setminus \overline{B_1}$ and $B_1$ and the normal flux $\nu \cdot X$ is continuous across $\partial B_1$ to show that $\sigma_A \nabla v$ is divergence-free. We have shown in Step (2) above that the vector field $\sigma(y,v)\nabla_{y}v$ is uniformly bounded away from $\partial B_2$ and by Step (3), its normal flux $\sigma_A(y,v) \nabla v \cdot \nu$ is continuous across $\partial B_1$. Moreover, it is obvious that $\sigma(y,v) \nabla _{y}v $ is divergence free in $B_1$ as $v$  is constant there. By Lemma \ref{lemma3.1}, $\sigma(y,v)\nabla_{y}v$ is weakly divergence free on $B_2 \setminus \overline{B_1}$ and thus we can conclude that $-div_y(\sigma_A(y,v) \nabla_y v) = 0 $ weakly in $B_2$.
\end{compactenum}
\end{proof}
We have shown in Lemma \ref{existence} that $v$ as defined in \eqref{defnv} solves \eqref{cloakingeqn}. We have also identified in Lemma \ref{lemma3.1} the values of $v$ in $B_2 \setminus B_1$. Since $\sigma_A$ is degenerate, uniqueness can fail. For instance, if $\sigma_A$ is identically $0$ in $B_1$ then $v$ can be arbitrary in $B_1$. However, such a possibility does not arise here as the degeneracy for $\sigma_A$ is only near $\partial B_1$. 

To show that $v = u(0)$ in $B_1$, we need to restrict further the class in which $v$ belongs. We assumed earlier that $v$ is uniformly bounded near $\partial B_1$. We need a condition to make $v$ continuous across $\partial B_1$ and a hypothesis on $\sigma_A(y,v(y)) \nabla v(y)$ for the PDE \eqref{cloakingeqn} to make sense.  We thus further assume
\begin{align}\label{uniquenessconditions}
	&\nabla v \in L^2(B_2) \mbox{ and }
	\sigma_A(y,v) \nabla v \in L^2(B_2)
	\end{align}
\begin{lemma}\label{lemma3.3}
If $v$ is a weak solution of \eqref{cloakingeqn} which satisfies \eqref{boundedness} and \eqref{uniquenessconditions}, then $v$ must be given by \eqref{defnv}.
\end{lemma}
\begin{proof}
By Lemma \ref{lemma3.1}, $v(y) = u(x)$ for $1<|y|<2$. By assumption \eqref{boundedness}, since $v(y)$ is uniformly bounded away from $|y|=1$, $u(x)$ has a removable singularity at $0$. In particular, it is continuous at $0$. As $F^{-1}$ maps $\partial B_1$ to $0$, $v(y) \to u(0)$ as $y$ approaches $\partial B_1$ from outside. 

Since $\nabla v \in L^2(B_2)$ by assumption, the trace of $v$ on $\partial B_1$ is well defined. Since $v(y) \to u(0)$ as $y$ approaches $\partial B_1$ from outside, the restriction of $v$ on $\partial B_1$ must be equal to $u(0)$.  By the uniqueness result in Theorem \ref{appendix} for the boundary value problem 
\begin{equation}
\begin{aligned}
	-div_{y} (A(y,v(y)\nabla_{y} v) &= 0 \mbox{ in } B_1\notag \\
v &= u(0) \mbox{ on }\partial B_1
\end{aligned}
\end{equation}
we can conclude that $v= u (0)$ in $B_1$.
\end{proof}
\subsection{Equality of DN maps}
We now show that the singular cloak $F^{*}_1$ cloaks $B_1$ in the sense of Definition 1. 
\begin{theorem}
Assume $\sigma_{A}(x,t)$ is as defined in \eqref{defnsigma}, where $F$ is given by \eqref{definitionoff} and $A \in \mathcal{M}(\alpha,\beta, L; \Omega \times \mathbb{R})$. Then the associated Dirichlet-to-Neumann map  $\Lambda_{\sigma_{A}}$ is equal to $\Lambda_{1}$.
\end{theorem}
\begin{proof}
By Lemma \ref{lemma3.1}, Lemma \ref{existence} and Lemma \ref{lemma3.3}, the Dirichlet-to-Neumann map $\Lambda_{\sigma_A}$ is well defined. Let $f,g \in H^{\frac{1}{2}}(\partial B_2)$. Let $\phi_{g} \in H^{1}(B_2)$ be such that $\phi_{g}|_{\partial B_2} = g$. We have
\begin{align}\label{perfect}
\langle \Lambda_{\sigma_A} f, g \rangle &= \int_{B_2 \setminus B_1} \sigma_A(y, v) \nabla_{y} v \cdot \nabla_{y} \phi_g \,dy \notag \\
& = \int_{B_2 \setminus \{0\}} \nabla_x u \cdot \nabla_x \tilde \phi_g \notag \\
&=\int_{B_2 } \nabla_x u \cdot \nabla_x \tilde \phi_g \notag \\ 
&=\langle \Lambda_1 f ,g \rangle
\end{align}
This completes the proof for perfect cloaking.
\end{proof}
We have focused on the radial setting because of its explicit character. The analysis extends similarly to non-radial domains.
\begin{corollary}
	Let $H:B_2 \to \Omega$ be a Lipschitz continuous map with Lipschitz continuous inverse, and suppose $E = H(B_1)$. Then $G = H \circ F \circ H^{-1}: \Omega \to \Omega$ is identity on $\Omega$ and $H$ "expands" the point $H(0)$ to $E$. Let
	\begin{align}\label{sigmaa1}
	\sigma_A(x,t)= \begin{cases} 
	A(x,t) & (x,t) \in E \times \mathbb{R}\\
	G_{*}1 & (x,t) \in \Omega \setminus E \times \mathbb{R} \\
	\end{cases}
	\end{align}
	where $A(x,t) \in \mathcal{M}(\alpha,\beta,L;\
	\Omega \times \mathbb{R})$. Then we have
	\begin{equation}\label{313}
	\Lambda_{\sigma_A} = \Lambda_{1}
	\end{equation}
	In other words, $ G_{*}1$ perfectly cloaks $E$. 
	\end{corollary}
\begin{proof}
	The proof in Theorem 4 in \cite{KSVW} goes through, almost word by word, for the quasi-linear equation \eqref{basiceqn}. For brevity, we omit the details and refer the reader to Theorem 4 in \cite{KSVW}.
	\end{proof}
Let us remark that \eqref{313} does not contradict the uniqueness up to diffeomorphism result of Sun and Uhlmann in \cite{SunUhlmann}. The result in \cite{SunUhlmann} crucially depends on the ellipticity of the matrix $A(x,t)$ in \eqref{basiceqn} which we violate by considering a singular change of variables that makes $A(x,t)$ degenerate. In other words, non-uniqueness for the Calder\'on problem for the quasi-linear elliptic equation \eqref{basiceqn} is possible if we allow $A(x,t)$ to be degenerate. 
 
\section{Homogenization Framework}
In Section 3, we showed how it is possible to nearly cloak $E \subset \Omega$. The approximate cloak, though non-degenerate, is anisotropic. What we would like to do in this section is to construct near cloaks which are isotropic. This will be done within the framework of homogenization. This section is organized as follows. In section 4.1, we develop the tools needed to prove homogenization for the quasi-linear PDE \eqref{basiceqn} with inhomogeneous boundary conditions for locally periodic microstructures. In section 4.2, we use the results of section 4.1 to construct explicit isotropic regular approximate cloaks for radial domains.

The main idea of homogenization
process \cite{A,T} is to provide a (macro scale) approximation to a problem with heterogeneities/microstructures (at micro scale) by suitably averaging out small scales and by incorporating their effects on large scales. These effects are quantified by the so-called homogenized coefficients.

Here we are concerned with the notion of $H$-convergence for quasi-linear PDEs of the form
\begin{equation}\begin{aligned}\label{ad1}
-div\left( A^\epsilon(x,u^\epsilon(x))\nabla u^\epsilon(x)\right) &= 0 \mbox{ in }\Omega\\
u^\epsilon &= f \mbox{ on }\partial\Omega.
\end{aligned}\end{equation}
We begin by recalling the notion of $H$-convergence \cite{A,T} in the linear case.  Let $\mathcal{M}(\alpha, \beta;\Omega)$ with
$0<\alpha<\beta$ denote the set of all real $N\times N$ symmetric matrices $A(x)$ of functions defined almost
everywhere on a bounded open subset $\Omega$ of $\mathbb{R}^N$ such that if $A(x)=[a_{kl}(x)]_{1\leq k,l\leq N}$ then 
\begin{equation*} a_{kl}(x)=a_{lk}(x)\ \forall l, k=1,..,N \ \mbox{and }\   (A(x)\xi,\xi)\geq \alpha|\xi|^2,\ |A(x)\xi|\leq \beta|\xi|,\  \forall\xi \in \mathbb{R}^N,\ \mbox{ a.e. }x\in\Omega.\end{equation*}
Let $A^\epsilon$ and $A^{*}$ belong to $\mathcal{M}(\alpha,\beta; \Omega)$. We say $A^\epsilon \xrightarrow{H} A^{*}$ or $H$-converges to 
a homogenized matrix $A^{*}$, if $A^\epsilon\nabla u^\epsilon \rightharpoonup A^{*}\nabla u$ in $L^2(\Omega)^N$ weak, for all test sequences $u^\epsilon$ satisfying 
\begin{align*}
u^{\epsilon} &\rightharpoonup u \quad\mbox{weakly in }H^1(\Omega)\\
-div(A^\epsilon\mathbb\nabla u^\epsilon)& \mbox{ is strongly convergent in } H^{-1}(\Omega).
\end{align*}
In particular, we consider the homogenization of  linear PDEs of the form 
\begin{equation*}\begin{aligned}
-div\left( A^\epsilon(x)\nabla u^\epsilon(x)\right) &= 0 \mbox{ in }\Omega\\
u^\epsilon &= f \mbox{ on }\partial\Omega
\end{aligned}\end{equation*}
where $f\in H^{\frac{1}{2}}(\partial\Omega)$. Then as $\epsilon \to 0$ we say $A^\epsilon \xrightarrow{H} A^{*}$, whenever
\begin{align*}
u^\epsilon &\rightharpoonup u \quad\mbox{weakly in }H^1(\Omega)\\ 
A^\epsilon\nabla u^\epsilon &\rightharpoonup A^{*}\nabla u\quad\mbox{weakly in }L^2(\Omega)^N
\end{align*}
where $u\in H^{1}(\Omega)$ solves
\begin{equation*}\begin{aligned}
-div\left( A^{*}(x)\nabla u(x)\right) &= 0 \mbox{ in }\Omega\\
u &= f \mbox{ on }\partial\Omega.
\end{aligned}\end{equation*}
\noindent
\textbf{Homogenization with periodic microstructures (linear case):}
Let us give an example in the class of periodic microstructures and its homogenization.
Let $Y$ denote the unit cube $[0,1]^N$ in $\mathbb{R}^N$.
Define $A:Y\mapsto \mathbb{R}^{N\times N}$ as $A(y)=[a_{kl}(y)]_{1\leq k,l\leq N}\in \mathcal{M}(\alpha,\beta; Y)$  such that $a_{kl}(y)$ are $Y$-periodic functions
$ \forall k,l =1,2..,N.$, which means $a_{kl}(y + z) = a_{kl}(y)$ whenever $z \in \mathbb{Z}^N$ and $y \in Y$. 
Now we set 
\[
A^{\epsilon}(x) = [a_{kl}^{\epsilon}(x)]= [a_{kl}(\frac{x}{\epsilon})]
\] and extend 
it to the whole $\mathbb{R}^N$ by $\epsilon$-periodicity with a small period of scale $\epsilon$ via scaling the coordinate $y=\frac{x}{\epsilon}$. The restriction of $A^{\epsilon}$ on $\Omega$ is known as periodic micro-structures.\\
In this classical case, the homogenized conductivity $A^{*}=[a^{*}_{kl}]$ is a constant matrix and can be defined by its entries (see \cite{A,BLP,CDDP}) as
\[
a^{*}_{kl} = \int_{Y}a_{ij}(y)\frac{\partial}{\partial y_i}(\chi_k(y) + y_k)\frac{\partial}{\partial y_j}(\chi_l(y) + y_l)dy
\]
where we define the $\chi_k$ through the so-called cell-problems. 
For each canonical basis vector $e_k$, consider the following conductivity problem in the periodic unit cell : 
\[
-div_y\ A(y)(\nabla_y\chi_k(y)+e_k) = 0 \quad\mbox{in }\mathbb{R}^N,\quad y \rightarrow\chi_k(y) \quad\mbox{is $Y$-periodic. }
\]
Let us generalize the above case and consider a locally periodic function $A:\Omega\times Y \mapsto \mathbb{R}^{N\times N}$defined  as  $A(x,y)=[a_{kl}(x,y)]_{1\leq k,l\leq N}\in \mathcal{M}(\alpha,\beta; \Omega\times Y)$  such that $a_{kl}(\cdot,y)$ are $Y$-periodic functions with respect to the second variable
$ \forall k,l =1,2..,N$ and for almost every $x$ in $\Omega$. Now we set 
\[
A^{\epsilon}(x) = [a_{kl}^{\epsilon}(x)]= [a_{kl}(x,\frac{x}{\epsilon})]
\]
Then the homogenized conductivity $A^{*}(x)=[a^{*}_{kl}(x)]$ is defined by its entries (see \cite{BLP,OL})
\begin{equation}\label{ub10}
a^{*}_{kl}(x) = \int_{Y}a_{ij}(x,y)\frac{\partial}{\partial y_i}(\chi_k(x,y) + y_k)\frac{\partial}{\partial y_j}(\chi_l(x,y) + y_l)dy
\end{equation}
where  $\chi_k(\cdot,y)\in H^1(Y)$ solves the following cell problem for almost every $x$ in $\Omega$:
\[
-div_y\ A(x,y)(\nabla_y\chi_k(x,y)+e_k) = 0 \quad\mbox{in }\mathbb{R}^N,\quad y \rightarrow\chi_k(x,y) \quad\mbox{is $Y$-periodic. }
\]
We end  the discussion on homogenization for the linear case by mentioning the following localization result \cite{A}.
\begin{proposition}\label{prop}
Let $A^\epsilon(x)$ H-converge to $A^{*}(x)$ in $\Omega$. Let $\omega$ be an open subset of $\Omega$. Then $A^\epsilon|_{\omega}$ (restrictions of $A^\epsilon$ to $\omega$) H-converge to $A^{*}|_{\Omega}$
	\end{proposition}
Based on the above localization result, we present the following example.

Let $\Omega$ be a domain which is subdivided into domains $\Omega^z$, $z=1,2...,m$ with Lipschitz boundaries. Let $A^z(x,y)$ be periodic functions in $y$ variable with periods $Y^z$ for $z=1,2,...,m$. 
Let us define for any $\epsilon>0$, 
\[
\widetilde{A}^{\epsilon}(x) = A^z(x,\frac{x}{\epsilon}) \mbox{ if } x \in \Omega^z.
\]
Then 
\begin{equation}\label{ub11}
\mbox{$\widetilde{A}^\epsilon(x)$ $H$-converges to $\widetilde{A}^{*}(x)$ in $\Omega$,}
\end{equation}
where  
\[
\widetilde{A}^{*}(x) = A^{*,z}(x) \mbox{ if } x \in \Omega^z
\]
and the $A^{*,z}(x)=[a^{*,z}_{kl}(x)]$ is defined as in \eqref{ub10} over the periodic cell $Y^z$:
\[
a^{*,z}_{kl}(x) =\frac{1}{|Y^z|} \int_{Y^z}a^z_{ij}(x,y)\frac{\partial}{\partial y_i}(\chi^z_k(x,y) + y_k)\frac{\partial}{\partial y_j}(\chi^z_l(x,y) + y_l)dy, \quad z=1,..,m
\]
where  $\chi^z_k(\cdot,y)\in H^1(Y^z)$ solves the following cell problem for almost every $x$ in $\Omega^z$:
\[
-div_y\ A^z(x,y)(\nabla_y\chi^z_k(x,y)+e_k) = 0 \quad\mbox{in }\mathbb{R}^N,\quad y \rightarrow\chi^z_k(x,y) \quad\mbox{is $Y^z$-periodic. }
\]

\vspace{5mm}
Let us now turn our attention to the equation \eqref{ad1}. For each fixed $\epsilon>0$, we consider $A^{\epsilon}(x,t) \in \mathcal{M}(\alpha,\beta,L;\Omega \times \mathbb{R})$ where $\alpha, \beta, L$ are positive, finite and independent of $\epsilon$.

It is shown in Theorem \ref{appendix}, that, for all fixed $\epsilon > 0$, the weak form of \eqref{ad1} with $f\in H^{\frac{1}{2}}(\partial\Omega)$ has a unique solution $u^\epsilon\in H^1(\Omega)$ satisfying the estimate
\begin{equation}\label{ad3}
||u^\epsilon||_{H^1(\Omega)} \leq C||f||_{H^{\frac{1}{2}}(\partial\Omega)}
\end{equation}
where $C=C(N,\Omega,\alpha,\beta,L)$ is independent of $\epsilon$.
Thus, standard compactness arguments imply that up to a subsequence (still denoted by $\epsilon$)
\[
u^\epsilon \rightharpoonup u \mbox{ weakly in } H^1(\Omega).
\]
Our goal is to get the limiting equation for $u\in H^1(\Omega)$. We remain in the class of periodic microstructures and derive the homogenization result in the quasi-linear settings. 
\subsection{Periodic homogenization for quasi-linear PDEs}
Let $Y$ denote the unit cube $[0,1]^N$ in $\mathbb{R}^N$.
Let $A:\Omega\times Y\times \mathbb{R}\mapsto \mathbb{R}^{N\times N}$ and $A(x,y,t)=[a_{ij}(x,y,t)]\in\mathcal{M}(\alpha,\beta,L;\Omega\times Y\times\mathbb{R})$ be such that 
\[y\mapsto  a_{ij}(x,y,t), \mbox{ are $Y$-periodic functions
for almost every $(x, t)\in \Omega\times\mathbb{R}$ and $i,j =1,2..,N.$} 
\]
We now set
\[A^{\epsilon}(x,t) = [a_{ij}(x,\frac{x}{\epsilon},t)],\quad (x,t)\in\Omega\times\mathbb{R}
\]
and this is known as periodic micro-structures in quasi-linear settings.

We will show that in this case the homogenized conductivity \\$A^{*}(x,t)=[a^{*}_{ij}(x,t)]\in \mathcal{M}(\widetilde{\alpha},\widetilde{\beta},\widetilde{L};\Omega\times\mathbb{R})$ can be defined by its entries (see \cite{Malik,FM,MURAT})
\begin{equation}\label{ad5}
a^{*}_{kl}(x,t) = \int_{Y} a_{ij}(x,y,t)\frac{\partial}{\partial y_i}(\chi_k(x,y,t) + y_k)\frac{\partial}{\partial y_j}(\chi_l(x,y,t) + y_l)dy
\end{equation}
where $\chi_k(x,y,t)\in H^1_{\#}(Y)$ for almost every $(x,t)\in\Omega\times \mathbb{R}$ are the solutions of the so-called cell-problems:
For each canonical basis vector $e_k\in\mathbb{R}^N$, $\chi_k(x,y,t)$ satisfy  the following problem for $y \in Y$, where $Y$ is the periodic unit cell and for almost every $(x,t)\in\Omega\times\mathbb{R}$:
\begin{equation}\begin{aligned}\label{ad6}
&-div_y\ A(x,y,t)(\nabla_y\chi_k(x,y,t)+e_k) = 0 \quad\mbox{in }\mathbb{R}^N,\\
\quad y &\rightarrow\chi_k(x,y,t) \quad\mbox{is $Y$-periodic for all $(x,t)\in\Omega\times \mathbb{R}$. }
\end{aligned}\end{equation}
The above problem has a unique solution in $H^1_{\#}(Y)/\mathbb{R}$ where,
\[
H^1_{\#}(Y)=\{f\in H^1_{loc}(\mathbb{R}^N):\ y\mapsto f(y) \mbox{ is $Y$ periodic.}\}
\] 
We further assume that,
\[
 \int_Y \chi_k(x,y,t)dy =0 \quad\forall (x,t)\in\Omega\times\mathbb{R}
\]
in order to get unique solution $\chi(x,y,t)\in H^1_{\#,0}(Y)$ for almost every $(x,t)\in\Omega\times\mathbb{R}$, where $H^1_{\#,0}(Y)=\{ f\in H^1_{\#}(Y): \int_Y f=0\}$.

Note that, from \eqref{ad5}
it follows that $a^{*}_{kl} =a^{*}_{lk}$ for $k,l=1,..,N$ and there exist $0<\widetilde{\alpha}<\widetilde{\beta}<\infty$ such that, $A^{*}(x,t)\in\mathcal{M}(\widetilde{\alpha},\widetilde{\beta},\widetilde{L};\Omega\times \mathbb{R})$. We will in fact show that $t \mapsto A^{*}(x,t)$ is uniformly Lipschitz in $\Omega$.

Before proving the homogenization result, we first discuss few properties of the expected homogenized matrix $A^{*}(x,t)$ defined in \eqref{ad5}.

\begin{lemma}\label{lemma4.1}
Let $A(x,y,t)\in\mathcal{M}(\alpha,\beta,L;\Omega\times Y\times\mathbb{R})$ be such that
\[y \mapsto A(x,y,t)=[a_{ij}(x,y,t)] \quad\mbox{is $Y$ periodic for almost every $(x,t)\in\Omega\times\mathbb{R}$}\] 
and $t\mapsto A(x,y,t)$ is uniformly Lipschitz for almost every $(x,y)\in\Omega\times Y$, i.e. 
\begin{equation}\label{ub8}
|a_{ij}(x,y,t_1)- a_{ij}(x,y,t_2)| \leq L|t_1-t_2| \quad i,j=1,..,N.
\end{equation}
Then the unique solution $\chi_k(x,\cdot,t)\in H^1_{\#,0}(Y)$, $k=1,...,N$ to \eqref{ad6} is such that  $t\mapsto \chi_k(x,y,t)$ is uniformly Lipschitz for almost every $(x,y)\in\Omega\times Y$ i.e.
\begin{equation}\label{53}
||\chi_k(x,\cdot,t_1) - \chi_k(x,\cdot,t_2)||_{H^1_{\#,0}(Y)} \leq C_{L} |t_1 - t_2|, \quad k=1,...,N
\end{equation}
and this holds for any $t_1, t_2 \in \mathbb{R}$, where $C_{L}$ is a constant independent of $x,y,t_1,t_2$.
\end{lemma}
\begin{proof}
By using the fact $A(x,y,t)\in \mathcal{M}(\alpha,\beta,L;\Omega\times Y\times \mathbb{R})$ it can be easily seen that, 
\begin{equation}\label{ub7}
||\chi_k(x,y,t)||_{H^1_{\#,0}}\leq C   \quad\mbox{almost every }(x,t)\in\Omega\times\mathbb{R}, \ \ k=1,..,N
\end{equation}
where $C$ is independent of $x,y,t,k$.
Let $\chi_k(x,y,t_1)\in H^1_{\#,0}(Y)$ and $\chi_k(x,y,t_2)\in H^1_{\#,0}(Y)$ be solutions to \eqref{ad6} for two pairs of points $(x,t_1)$ and $(x,t_2)$ in $\Omega\times \mathbb{R}$, then after simple operations we have the relations
\begin{align*}
&\int_Y A(x,y,t_1) \nabla\left(\chi_k(x,y,t_1)-\chi_k(x,y,t_2)\right)\cdot\nabla\left(\chi_k(x,y,t_1)-\chi_k(x,y,t_2)\right)dy \notag\\
&=\int_Y \left(A(x,y,t_2) - A(x,y,t_1)\right) \left(\nabla\chi_k(x,y,t_2)+ e_k\right)\cdot \nabla\left(\chi_k(x,y,t_1)-\chi_k(x,y,t_2)\right)dy.\end{align*}
Now by using the coercivity of $A$ and together with \eqref{ub7}, the Lipschitz criteria \eqref{ub8} we have 
\begin{align}\label{58}
&\alpha ||\chi_k(x,\cdot,t_1) - \chi_k(x,\cdot,t_2)||^2_{H^1_{\#,0}(Y)}  \notag \\
&\leq\int_Y A(x,y,t_1) \nabla\left(\chi_k(x,y,t_1)-\chi_k(x,y,t_2)\right)\cdot\nabla\left(\chi_k(x,y,t_1)-\chi_k(x,y,t_2)\right)dy \notag\\
&=\int_Y \left(A(x,y,t_2) - A(x,y,t_1)\right) \left(\nabla\chi_k(x,y,t_2)+ e_k\right)\cdot \nabla\left(\chi_k(x,y,t_1)-\chi_k(x,y,t_2)\right)dy\notag\\
& \leq \left( \sum_{i,j=1}^n ||a_{ij}(x,y,t_1) - a_{ij}(x,y,t_2)||_{L^{\infty}(Y)} \right) \left(||\chi_k(x,\cdot,t_2)||_{H^{1}_{\#,0}(Y)} + \sqrt{|Y|} \right)\notag\\
&\qquad\qquad\qquad\qquad\qquad\qquad\qquad\qquad\qquad\cdot ||\chi_{k}(x,\cdot,t_1) - \chi_k(x,\cdot,t_2)||_{H^1_{0,\#}(Y)} \quad k=1,2...,n
\end{align}
 \eqref{53} thus follows from \eqref{58}.
\end{proof}
Recalling the definition of $A^{*}=[a^{*}_{ij}(x,t)]$, (see \eqref{ad5}), then following  Lemma \ref{lemma4.1} we have the following result.  
\begin{lemma}\label{lemma4.2}
Let the assumptions of Lemma \ref{lemma4.1} be satisfied. Then $t\mapsto A^{*}(x,t)$ is uniformly Lipschitz, i.e.
\begin{equation}\label{kaha}
|a_{ij}^{*}(x,t_1) - a_{ij}^{*}(x,t_2) | \leq \widetilde{L}|t_1 - t_2|, \quad i,j=1,2,...,n
\end{equation}
hold for almost every $x\in\Omega$ and $t_1,t_2 \in \mathbb{R}$  where $\widetilde{L}$ is independent of $x,t_1,t_2$. 
\end{lemma}
\begin{proof}
Let us write \eqref{ad5} as 
\[
a^{*}_{kl}(x,t) = \langle a_{il}(x,y,t)\frac{\partial}{\partial y_i}\chi_k(x,y,t)\rangle + \langle a_{kl}(x,y,t)\rangle
\]
where $\langle \rangle$ denotes the average over the periodic cell $Y$.

Let us analyze the first term in the above  formula. We write
\begin{align*}
&| \left \langle a_{il}(x,\cdot,t_1) \frac{\partial \chi_k}{\partial y_i}(x,\cdot,t_1) - a_{il}(x,\cdot,t_2) \frac{\partial \chi_k}{\partial y_i}(x,\cdot,t_2) \right \rangle |  \notag \\
&\leq| \left \langle a_{il}(x,\cdot,t_2) \left(\frac{\partial \chi_k(x,\cdot,t_1)}{\partial y_i} - \frac{\partial \chi_k(x,\cdot,t_2)}{\partial y_i} \right ) \right \rangle | +\notag \\
&| \left \langle \left( a_{il}(x,\cdot,t_1) - a_{il}(x,\cdot,t_2)\right) \frac{\partial \chi_k(x,\cdot,t_1)}{\partial y_i}  \right \rangle |.
\end{align*}
Then by using the fact that $t\mapsto a_{il}(x,y,t)$ and $t\mapsto \chi_k(x,y,t)$ are uniformly Lipschitz functions, it follows that $t\mapsto a^{*}_{il}(x,y,t)$ is uniformly Lipschitz for almost every $(x,y)\in \Omega\times Y$. Hence we have \eqref{kaha}.
\end{proof}
Next we present the local characterization in the quasi-linear settings analogous to local case in the linear setting mentioned in Proposition \ref{prop}.
\begin{lemma}\label{lemma4.3}
Let $A^\epsilon(x,t)$ governed by the periodic microstructures $H$-converge to $A^{*}(x,t)$ in $\Omega\times\mathbb{R}$. Let $\omega$ be an open subset of $\Omega$ and assume that $A^\epsilon|_{\omega\times\mathbb{R}}$ (restrictions of $A^\epsilon$ to $\omega\times\mathbb{R}$) are independent of $t$ i.e
\[
A^\epsilon(x,t) = A^\epsilon(x) \quad\mbox{whenever }(x,t)\in\omega\times\mathbb{R}.
\]
Then the homogenized limit $A^{*}(x,t)|_{\omega\times\mathbb{R}}$ is also independent of $t$, i.e.
\[
A^{*}(x,t) = A^{*}(x) \quad\mbox{whenever, }(x,t)\in\omega\times\mathbb{R}.
\]
\end{lemma}
\begin{proof}
The proof is straightforward from the locality of the second order PDE \eqref{ad6} satisfied by $\chi_k(x,y,t)$; as under the assumption $\chi_k(x,y,t) =\chi_k(x,y)$ whenever $(x,y,t)\in\omega\times Y\times\mathbb{R}$ and by using \eqref{ad5} we can conclude that $A^{*}(x,t) = A^{*}(x)$ whenever $(x,t)\in\omega\times\mathbb{R}$.
\end{proof}

Now we are going to prove  homogenization result for the quasi-linear PDE with inhomogeneous boundary conditions for locally  periodic microstructures. Similar result for example, with homogeneous boundary condition and globally periodic microstructures can be found in \cite{Malik,FM}. We first choose $\Omega$ to be a smooth enough domain and later relax the assumptions and let $\Omega$ be a Lipschitz domain.

\begin{theorem}\label{theorem4.1}
Let $\Omega$ be a bounded domain in $\mathbb{R}^N$ with $C^{2,\gamma}$ boundary where $0 < \gamma <1$.  Let the matrix $A(x,y,t)$ satisfy:
\begin{enumerate}
\item $A(x,y,t)\in\mathcal{M}(\alpha,\beta,L;\Omega\times Y\times\mathbb{R})$, where $\alpha,\beta,L$ are independent of $x,y,t$.
\item $(x,y,t) \mapsto A(x,y,t)=[a_{ij}(x,y,t)] \quad\mbox{is }C^{1,\gamma}(\overline{\Omega}\times Y\times\mathbb{R})$.
\end{enumerate}
Let us consider the sequence of matrices $\{A^\epsilon(x,t)\}_{\epsilon>0}$ for $(x,t)\in\Omega\times\mathbb{R}$ given by 
\[ A^\epsilon(x,t) := A(x,\frac{x}{\epsilon},t)\] and consider the following inhomogeneous quasi-linear PDEs with $f\in C^{2,\gamma}(\overline{\Omega})$:
\begin{equation}\label{ub3}
\begin{aligned}
-div\left( A^\epsilon(x,u^\epsilon)\nabla u^\epsilon(x)\right) &= 0 \mbox{ in }\Omega\\
u^\epsilon &= f \mbox{ on }\partial\Omega.
\end{aligned}
\end{equation}     
Then upto a subsequence the corresponding solutions $\{u^\epsilon\}_{\epsilon}\in H^1(\Omega)$ of \eqref{ub3} satisfy
\begin{align*}
u^\epsilon &\rightharpoonup u \quad\mbox{weakly in }H^1(\Omega)\\ 
A^\epsilon(x,u^\epsilon)\nabla u^\epsilon &\rightharpoonup A^{*}(x,u)\nabla u\quad\mbox{weakly in }L^2(\Omega)^N;
\end{align*}
where $u\in H^1(\Omega)$ is the unique solution of the so-called homogenized problem
\begin{equation*}\begin{aligned}
-div\left( A^{*}(x,u(x))\nabla u(x)\right) &= 0 \mbox{ in }\Omega\\
u &= f \mbox{ on }\partial\Omega.
\end{aligned}\end{equation*}
We say $A^\epsilon(x,t) \xrightarrow{H} A^{*}(x,t)$ in $\Omega\times\mathbb{R}$, where the homogenized matrix $A^{*}(x,t)$ is defined as in \eqref{ad5}. 
\end{theorem}
\begin{proof}
For $\epsilon>0$ fixed, let us consider $f\in C^{2,\gamma}(\overline{\Omega})$, then the problem \eqref{ub3} has a unique solution $u^\epsilon\in C^{2,\gamma}(\overline{\Omega})$ satisfying (see Lemma \ref{ub4}), for $x,y\in\overline{\Omega}$
\[
 |u^\epsilon(x) -u^\epsilon(y)|\leq C|x-y|^{\lambda}
\]
where $C=C(N,\Omega,\alpha,\beta,f)$ and $\lambda=\lambda(N,\Omega,\alpha,\beta,\gamma)$ are independent of $\epsilon$. Thus it follows that the sequence $\{u^\epsilon\}_\epsilon$  satisfies the assumption of the
Arzel\'{a}-–Ascoli theorem. Hence there exists a subsequence still denoted by $\{u^\epsilon\}_\epsilon$ that converges strongly to $u$ in $C(\overline{\Omega})$. We also know that subsequence $\{u^\epsilon\}_\epsilon$ satisfies (see \eqref{ad3})
\[||u^\epsilon||_{H^1(\Omega)}\leq C||f|_{\partial \Omega}||_{H^{\frac{1}{2}}(\partial\Omega)}\] 
where $C$ is independent of $\epsilon$. Thus it has a weakly convergent subsequence still denoted as  $\{u^\epsilon\}_\epsilon$ in $H^1(\Omega)$ and the subsequential limit is same as $u\in H^1(\Omega)\cap C(\overline{\Omega})$.

Let $\delta >0$ be arbitrary, then the domain $\Omega$ can be divided in to sub-domains 
$\Omega^z$, $z=1,2...,m$ with Lipschitz boundaries and there exists a function $u^\delta$ constant on every sub-domain $\Omega^z$ such that for sufficiently small $\epsilon$, we have
\begin{equation}\label{616}
|u^{\epsilon}(x) - u^\delta(x)| \leq \delta \quad x \in \Omega
\end{equation}
 This immediately implies 
\begin{equation}\label{ad2}
|u(x) - u^{\delta}(x)| \leq \delta \quad x \in \Omega
\end{equation}
Let $u^{\delta,\epsilon} \in H^1(\Omega)$ be a sequence of solutions to the linear PDEs 
\begin{equation}\begin{aligned}\label{ad13}
-div\left( A^\epsilon(x,u^\delta(x))\nabla u^{\delta,\epsilon}(x)\right) &= 0 \mbox{ in }\Omega\\
u^{\delta,\epsilon} &= f \mbox{ on }\partial\Omega.
\end{aligned}\end{equation}
Then we have the following identity which follows from  \eqref{ub3} and \eqref{ad13}.
\begin{align}\label{619}
\int_{\Omega} A^{\epsilon}(x,u^\epsilon) \nabla\left(   u^\epsilon- u^{\delta,\epsilon}\right)\cdot \nabla\left(u^\epsilon- u^{\delta,\epsilon}\right)dx& \notag \\
=\int_{\Omega} &\left( A^\epsilon(x,u^{\delta}) - A^\epsilon(x,u^\epsilon)\right) \nabla u^{\delta,\epsilon}\cdot\nabla\left( u^\epsilon- u^{\delta,\epsilon}\right)dx
\end{align}
Using ellipticity for $A^{\epsilon}$ gives us
\begin{align}\label{620}
\alpha ||u^{\epsilon} - u^{\delta,\epsilon}||^2_{H^1_{0}(\Omega)}  &\leq\int_{\Omega} A^{\epsilon}(x,u^\epsilon) \nabla\left(   u^\epsilon- u^{\delta,\epsilon}\right)\cdot \nabla\left(u^\epsilon- u^{\delta,\epsilon}\right)dx \notag \\
&=\int_{\Omega} \left( A^\epsilon(x,u^{\delta}) - A^\epsilon(x,u^\epsilon)\right)\cdot \nabla u^{\delta,\epsilon}\cdot\nabla\left( u^\epsilon- u^{\delta,\epsilon}\right)dx \notag\\
&\leq\left( \sum_{i,j=1}^n||a_{ij}^{\epsilon}(x,u^\delta) - a_{ij}^{\epsilon}(x,u^{\epsilon})||_{L^{\infty}(\Omega)}\right)||u^{\delta,\epsilon}||_{H^1(\Omega)}||u^{\epsilon} - u^{\delta,\epsilon}||_{H^1_0(\Omega)}
\end{align}
Using the fact that $u^{\delta,\epsilon}\in H^{1}(\Omega)$ solves \eqref{ad13} with 
$||u^{\delta,\epsilon}||_{H^1(\Omega)} \leq C||f||_{H^{\frac{1}{2}}(\partial\Omega)}$
where $C$ is independent of $\delta$ and $\epsilon$, from \eqref{620} we obtain 
\begin{equation}\label{ad14}
  ||u^{\epsilon} - u^{\delta,\epsilon}||_{H^1_{0}(\Omega)} \leq C \left( \sum_{i,j=1}^n||a_{ij}^{\epsilon}(x,u^\delta) - a_{ij}^{\epsilon}(x,u^{\epsilon})||_{L^{\infty}(\Omega)}\right).
\end{equation}
As $A^\epsilon(\cdot,t)$ is uniformly Lipschitz in $x$ and $\epsilon$, therefore using \eqref{616} gives us 
\begin{align*} \left( \sum_{i,j=1}^n||a_{ij}^{\epsilon}(x,u^\delta) - a_{ij}^{\epsilon}(x,u^{\epsilon})||_{L^{\infty}(\Omega)}\right) \leq L\delta.\end{align*}
Consequently, from \eqref{ad14} it follows that, 
\begin{equation}\label{lb1}
 ||u^{\epsilon} - u^{\delta,\epsilon}||_{H^1_{0}(\Omega)} \leq C\delta.
\end{equation}
where $C$ is independent of $\epsilon$ and $\delta$ and this inequality holds for sufficiently small $\epsilon$.

Since $u^\delta(x)$ are constant in each $\Omega^z$ for $z=1,2...,m$, then applying the $H$-convergence result of \eqref{ub11} where we take $\widetilde{A^{\epsilon}}(x) = A^{\epsilon}(x, u^{\delta}(x)) = A(x, \frac{x}{\epsilon},u^{\delta}(x)) = \widetilde{A}_{\delta}(x,\frac{x}{\epsilon})$ in $\Omega^z$, we have
\[
\widetilde{A}_{\delta}(x,\frac{x}{\epsilon}) \mbox{ $H$-converges to } A^{*}_{\delta}(x) \quad\mbox{in }\Omega^z, \quad z=1,2,...,m.
\]
where $A^{*}_{\delta}(x)=A^{*}(x,u^{\delta}(x))$ is defined as in \eqref{ad5} with constant $t = u^{\delta}(x)$ for $x \in \Omega ^z, z=1,2,...,m$. Hence in the whole domain $\Omega$ we have 
\[
 \mbox{$A^{\epsilon}(x,u^{\delta}(x))$ H-converges to $A^{*}(x,u^{\delta}(x))$ in $\Omega$.}
\]
This means that we have 
\begin{equation}\label{lb2}
\begin{aligned}
u^{\delta,\epsilon} &\rightharpoonup u^{\delta,*} \mbox{ in }H^{1} (\Omega)  \\
A^\epsilon(x,u^\delta)\nabla u^{\delta,\epsilon} &\rightharpoonup A^{*}(x,u^\delta)\nabla u^{\delta,*} \mbox{ in }  L^{2}(\Omega)^N.
\end{aligned}
\end{equation}
where,  $u^{\delta,*} \in H^1(\Omega)$ is the  solution to the linear PDE
\begin{equation}\begin{aligned}\label{624}
-div\left( A^{*}(x,u^\delta(x))\nabla u^{\delta,*}(x)\right) &= 0 \mbox{ in }\Omega\\
u^{\delta,*} &= f \mbox{ on }\partial\Omega.
\end{aligned}\end{equation}
Let $w \in H^1(\Omega)$ be the solution to linear PDE
\begin{equation}\begin{aligned}\label{ad17}
-div\left( A^{*}(x,u(x))\nabla w(x)\right) &= 0 \mbox{ in }\Omega\\
w &= f \mbox{ on }\partial\Omega.
\end{aligned}\end{equation}
Then similar to \eqref{619} we have the following identity which follows from the equations \eqref{624} and \eqref{ad17}:
\begin{align*}
\int_{\Omega} A^{*}(x,u^\delta) \nabla\left(   w- u^{\delta,*}\right)\cdot \nabla\left(w- u^{\delta,*}\right)dx& \notag \\
=\int_{\Omega} &\left( A^{*}(x,u) - A^{*}(x,u^\delta)\right) \nabla u^{\delta,*}\cdot\nabla\left( w- u^{\delta,*}\right)dx.
\end{align*}
Then together with \eqref{ad2} and the fact that $A^{*}(\cdot,t)$ is uniformly Lipschitz in $x$ we get the following estimate
\begin{align}\label{lb3}
||u^{\delta,*} - w||_{H^1_0(\Omega)} \leq C \delta
\end{align}
where $C$ is independent of $\delta$.

Combining  \eqref{lb1}, \eqref{lb2} and \eqref{lb3} and by writing 
\[
\langle u^{\epsilon} - w,v\rangle = \langle(u^{\epsilon} - u^{\delta,\epsilon}),v\rangle + \langle(u^{\delta,\epsilon} - u^{\delta,*}),v\rangle + \langle (u^{\delta,*} - w^{*}),v\rangle
\]
where $\langle,\rangle,$ denotes the usual scalar product on $H^1(\Omega)$ and $v$ in $H^1(\Omega)$ is arbitrary, we obtain  
\begin{equation}\label{630}
|\langle u^{\epsilon}-w,v\rangle | \leq C\delta  ||v||_{H^1(\Omega)}
\end{equation}
where $C$ is independent of $\epsilon,\delta,v$ and the estimate holds for sufficiently small $\epsilon$. From \eqref{630}, since $\delta>0$ is arbitrary we can conclude that as $\epsilon \to 0$ 
\[
u^{\epsilon} \rightharpoonup w \mbox{ in } H^1(\Omega) 
\]
and by uniqueness of the weak limit we have $w=u$. 
Hence we obtain the homogenized equation as  
\begin{equation}\begin{aligned}\label{lb4}
-div\left( A^{*}(x,u(x))\nabla u(x)\right) &= 0 \mbox{ in }\Omega\\
u &= f \mbox{ on }\partial\Omega.
\end{aligned}\end{equation}
We note that as $A^{*}(\cdot,t)$ is uniformly Lipschitz in $x$ (By Lemma \ref{lemma4.2}), so from Theorem \ref{appendix}, it follows that the above problem \eqref{lb4} has the unique solution $u\in H^{1}(\Omega)$. 

Next we prove the the $L^2$ weak limit of the flux  $A^\epsilon(x,u^\epsilon)\nabla u^\epsilon$ is $A^{*}(x,u)\nabla u$. Let us consider the following expression
\begin{align}\label{632}
||A^\epsilon(x,u^\epsilon)\nabla u^\epsilon - A^\epsilon(x,u^\delta)\nabla u^{\delta,\epsilon}||_{L^2(\Omega^N)}\leq\ & ||A^{\epsilon}(x,u^{\epsilon})\nabla \left(u^\epsilon- u^{\delta,\epsilon}\right)||_{L^2(\Omega^N)}\notag\\
&\qquad+ ||\left(A^\epsilon(x,u^\epsilon) - A^{\epsilon}(x,u^\delta)\right)\nabla u^{\delta,\epsilon}||_{L^2(\Omega^N)}.
\end{align}
Then from the right hand side of \eqref{632} we get 
\begin{align}\label{633}
||A^{\epsilon}(x,u^{\epsilon})\nabla \left(u^\epsilon- u^{\delta,\epsilon}\right)||_{L^2(\Omega^N)}&\leq \left(\sum_{j=1}^n ||a_{ij}(x,u^\epsilon)||_{L^{\infty}(\Omega)} \right)||u^\epsilon - u^{\delta,\epsilon}||_{H^1(\Omega)}\notag \\
&\leq C\delta.
\end{align}
and using the Lipschitz continuity of $A^\epsilon(\cdot,t)$ we get 
\begin{align}\label{455}
||\left(A^\epsilon(x,u^\epsilon) - A^{\epsilon}(x,u^\delta)\right)\nabla u^{\delta,\epsilon}||_{L^2(\Omega^N)}&\leq \left( \sum_{j=1}^n ||a_{ij}^{\epsilon}(x,u^\epsilon) - a_{ij}^{\epsilon}(x,u^\delta)||_{L^{\infty}(\Omega)}\right)||u^{\delta,\epsilon}||_{H^1(\Omega)}\notag \\
&\leq C\delta.
\end{align}
Thus from \eqref{632}, \eqref{633} and \eqref{455} we have 
\begin{equation}\label{lb8}
||A^\epsilon(x,u^\epsilon)\nabla u^\epsilon - A^\epsilon(x,u^\delta)\nabla u^{\delta,\epsilon}||_{L^2(\Omega^N)}\leq C\delta.
\end{equation}
where $C$ is independent of $\epsilon$, $\delta$ and the estimates hold for sufficiently small $\epsilon$. \\
In a similar manner, we can prove the estimates 
\begin{align}\label{lb9}
||A^{*}(x,u^\delta)\nabla u^{\delta,*} - A^{*}(x,u)\nabla u||_{L^2(\Omega^N)} \leq C\delta.
\end{align}
where $C$ is independent of $\delta$. Finally we consider
\begin{align}\label{638}
\langle\left( A^\epsilon(x,u^\epsilon)\nabla u^\epsilon - A^{*}(x,u)\nabla u\right), v\rangle
&= 
\langle\left( A^\epsilon(x,u^\epsilon)\nabla u^\epsilon - A^{\epsilon}(x,u^\delta)\nabla u^{\delta,\epsilon}\right), v\rangle + \notag \\
&\quad \quad\langle\left( A^\epsilon(x,u^\delta)\nabla u^{\delta,\epsilon} - A^{*}(x,u^\delta)\nabla u^{\delta,*}\right), v\rangle +  \notag\\
&\quad\quad\langle\left( A^{*}(x,u^\delta)\nabla u^{\delta,*} - A^{*}(x,u)\nabla u\right), v\rangle.
\end{align}
here $\langle,\rangle$ denotes the usual inner product on $L^2(\Omega)^N$ and $v \in L^2(\Omega)^N$ is arbitrary. Now by using  \eqref{lb2}, \eqref{lb8}, and \eqref{lb9} we conclude from \eqref{638} that 
\begin{equation}\label{sb2}
|\langle\left( A^\epsilon(x,u^\epsilon)\nabla u^\epsilon - A^{*}(x,u)\nabla u\right), v\rangle| \leq C\delta.
\end{equation}
where $C$ is independent of $\epsilon, \delta$ and $v\in L^2(\Omega)^N$. This inequality holds for sufficiently small $\epsilon$. If we consider $\delta>0$ to be arbitrary, then \eqref{sb2} yields
\[A^\epsilon(x,u^\epsilon)\nabla u^\epsilon \rightharpoonup A^{*}(x,u)\nabla u \quad\mbox{ weakly in }L^2(\Omega)^N.\]
This completes the discussion of our proof. 
\end{proof}
\begin{remark}
Let us now consider a Lipschitz domain $\Omega^\prime$ such that $\Omega \subset \subset \Omega^\prime$. We extend the homogenization co-efficients $A^{\epsilon}(x,t)$  defined in Theorem \ref{theorem4.1} by identity in $\Omega^{\prime} \setminus \Omega \times \mathbb{R}$ and still denote the extended coefficients  $A^{\epsilon}$. Fix $f \in H^{\frac{1}{2}}(\partial \Omega^{\prime})$. We consider the following equation 
\begin{equation*}
\begin{aligned}
-div(A^{\epsilon}(x,u^{\epsilon}) \nabla u^{\epsilon}) &= 0 \mbox{ in } \Omega^\prime \\
u^{\epsilon}&= f \mbox{ on } \partial \Omega^\prime
\end{aligned}
\end{equation*}

Note that by interior regularity $u^{\epsilon} \in C^{2,\gamma}(\overline{\Omega}) \cap H^{1}(\Omega^\prime)$. 

Since the homogenization co-efficients $A^{\epsilon}(x,t) = I$ in the ring $\Omega^{\prime} \setminus \Omega \times \mathbb{R}$ for any $\epsilon>0$,  the conclusion of Theorem \ref{theorem4.1} still holds and passing through the limit in $\epsilon$ we obtain 
\[
A^{\epsilon}(x,t)\mbox{ H-converges to }A^{*}(x,t) \mbox{ in } \Omega^\prime \times \mathbb{R}.
\]
where, by localization principle, $A^{*}(x,t) = I$ in $\Omega^\prime \setminus \Omega \times \mathbb{R}$ and $A^{*}(x,t)$ in $\Omega \times \mathbb{R}$ is  in \eqref{ad5}.

Moreover, we also have 
\begin{align*}
u^\epsilon &\rightharpoonup u \quad\mbox{weakly in }H^1(\Omega^\prime)\\ 
A^\epsilon(x,u^\epsilon)\nabla u^\epsilon &\rightharpoonup A^{*}(x,u)\nabla u\quad\mbox{weakly in }L^2(\Omega^\prime)^N;
\end{align*}
where $u\in H^1(\Omega^\prime)$ is the unique solution to
\begin{equation*}\begin{aligned}
-div\left( A^{*}(x,u(x))\nabla u(x)\right) &= 0 \mbox{ in }\Omega^\prime\\
u &= f  \mbox{ on }\partial \Omega^\prime.
\end{aligned}
\end{equation*}
%In other words, the conclusion of Theorem \ref{theorem4.1}
%goes through for a Lipschitz domain as well.
\end{remark}

\subsection{Regular isotropic approximate cloak in $\mathbb{R}^N$ }
In this section we approximate the anisotropic approximate (or near) cloaks 
$\sigma_A^r(x,t)$ as defined in \eqref{sigmaar}  by isotropic conductivities, which then will themselves be approximate
cloaks. We restrict our attention to the case when $\Omega = B_2$ and $E=B_1$ needs to be cloaked.

We will be considering the isotropic conductivities of the  form
\begin{equation}\label{ad7}
A^\epsilon(x,t)= \sigma(x,\frac{|x|}{\epsilon},t)I ,\quad(x,t)\in\Omega\times\mathbb{R}.
\end{equation}
%where $r:=r(x)=|x|$ is the radial coordinate and
where  $\sigma(x,r^\prime,t)$ is a smooth,
scalar valued function such that 
\begin{enumerate}
	\item[A1)] $0 < \alpha \leq \sigma(x,r^\prime,t) \leq \beta < \infty$ for all $(x,r',t) \in \Omega \times [0,1] \times \mathbb{R}$.
	\item[A2)]
	$t \mapsto \sigma(x,r',t)$ is uniformly Lipschitz for $(x,r') \in \Omega \times [0,1]$.
	\item[A3)]
	$\sigma$ is periodic in $r^\prime$ with period $1$ i.e
	$\sigma(x,r^\prime+1,t)=\sigma(x,r^\prime,t), \quad (x,r^\prime,t)\in\Omega\times[0,1]\times\mathbb{R}$.
\end{enumerate}

We would like to find the homogenized coefficient $A^{*}(x,t)$ given in \eqref{ad5} for this choice of $A^\epsilon(x,t)$ in \eqref{ad7}. In order to do that, we introduce polar coordinates.
Let
$s_1=(r,\theta_1,..
,\theta_N)$ and $s_2=(r^\prime,\theta^\prime_1,..,\theta^\prime_N)$ be spherical coordinates
corresponding to
 two different scales.
Next we homogenize the conductivity
in the $(r^\prime,\theta^\prime_1,..,\theta^\prime_N)$-coordinates.
Let us consider the canonical basis vectors $e_1,.., e_N$ of $\mathbb{R}^N$  in $r^\prime$, $\theta^\prime_1,...,\theta^\prime_N$ directions,
respectively. Then for almost every $(s_1,t)\in\Omega\times \mathbb{R}$, let $\chi_k(s_1,s_2,t),\, k=1,..,N$ be
the solutions of
\begin{equation}\label{ad8}\begin{aligned}
&-div_{s_2}(\sigma(s_1,r^\prime,t)(\nabla_{s_2}\chi_k(s_1,s_2,t)+e_k)=0
\mbox{ in } \mathbb{R}^N,\\
 &s_2=(r^\prime,\theta^\prime_1,..,\theta^\prime_N) \rightarrow\chi_k(s_1,s_2,t) \quad\mbox{is $1$-periodic in each of coordinate $r^\prime,\theta_1^\prime,..,\theta_N^\prime$. }
\end{aligned}\end{equation}
Further, it is assumed that
\begin{equation}\label{ad9}
\int_ {Y} \chi_k(s_1,s_2^\prime,t)ds_2^\prime =0,
\end{equation}
where, $s_2^\prime=(r^\prime,\theta^\prime_1,..,\theta^\prime_N)$ and $ds_2^\prime =dr^\prime d\theta^\prime_1...d\theta^\prime_N$.

Since $\sigma(s_1,r^\prime,t)$ is independent of $\theta^\prime_1,..,\theta^\prime_N$, therefore from \eqref{ad8} together with \eqref{ad9}, it implies that $\chi_k=0$ for $k=2,..,N$.
Let us consider the equation \eqref{ad8} for $\chi_1$ which satisfies
\begin{equation}\label{ad11}
\frac{\p}{\p r^\prime} \left( \sigma(s_1, r^\prime,t)\frac{\p \chi_1(s_1,s_2,t)}{\p r^\prime}\right)=
-\frac{\p  \sigma(s_1, r^\prime,t)}{\p r^\prime}.
\end{equation}
From the fact that $\chi_1$ is $1$-periodic with respect to the variables  $\theta^\prime_1,...,\theta^\prime_n$ it implies that $\chi_1$ is independent of $\theta^\prime_1,...,\theta^\prime_n$ Moreover, from \eqref{ad11} we get
\begin{equation}
\frac{\p \chi_1}{\p r^\prime} =-1 + \frac{C}{\sigma(s_1, r^\prime,t)}.
\end{equation}
where the constant $C$ can be found by using the periodicity of $\chi_1$  with respect to $r^\prime$ as 
\[
C =\frac{1}{\int_0^1\sigma^{-1}(s_1, r^\prime,t)dr^\prime } :=\underline{\sigma}(s_1,t),
\]
where $\underline{\sigma}(s_1,t)$ denotes the harmonic mean of $\sigma(s_1,\cdot,t)$ in the second variable.

Let  $\overline{\sigma}(s_1,t)$ denote the arithmetic mean of $\sigma(s_1,\cdot,t)$ in the second variable as:
\[\overline{\sigma}(s_1,t) =\int_0^1 \sigma(s_1,r^\prime,t)dr^\prime.
\]
Then from \eqref{ad5} the homogenized conductivity, say  $\widehat{\sigma}_{kl}(s_1,t)$, turns out to be 
\[
 \widehat{\sigma}_{kl}(s_1,t) = \int_Y \sigma(s_1,r^\prime,t)\left(\frac{\partial\chi_k(s_1,s_2^\prime,t)}{\partial (s_2^\prime)_l} + \delta_{lk}\right)ds_2^\prime
\]
and can be written as 
\[
\widehat{\sigma}(x,t)=\underline{\sigma}(x,t)\Pi(x)+\overline{\sigma}(x,t)(I-\Pi(x)),
\]
where $\Pi(x):\mathbb{R}^N\to \mathbb{R}^N$ is the projection on to the radial direction, defined by
\[
\Pi(x)\,v=\left(v\,\cdot \frac{x}{|x|}\right)\frac{x}{|x|},
\]
i.e., $\Pi(x)$ is represented by the matrix $|x|^{-2}xx^t$, cf. \cite{GYLU}. 

Next we give more explicit construction of the regular isotropic cloak.% for the case when a ball needs to be cloaked i.e the domain and the region to be cloaked are spherically symmetric. 
This construction is a generalization of the linear case presented in \cite{GYLU}.  Let us consider functions $\phi:\mathbb{R}\to \mathbb{R}$ and $\phi_M:\mathbb{R}\to \mathbb{R}$ given by 
\begin{eqnarray} \label{phi}
\phi(t)=\left\{\begin{array}{cl}
0,& t<0,\\
\frac 12 t^2,& 0\leq t<1,\\
1- \frac 12 (2-t)^2,& 1\leq t<2\\
1,& t\geq 2,\end{array}
\right.
\end{eqnarray} 
and
\begin{eqnarray} \label{phim}
\phi_M(t)=\left\{\begin{array}{cl}
0,& t<0,\\
\phi(t),& 0\leq t<2,\\
1,& 2\leq t<L-2,\\
\phi(M-t),& t\geq L-2.\end{array}
\right.
\end{eqnarray} 
Let us define
\[
\sigma(x,r^\prime,t) =\left[1+a^1(x,t)\zeta_1(\frac{r}{\epsilon})-a^2(x,t)\zeta_2(\frac{r}{\epsilon})\right]^2, 
\]
where $a^k(x,t)$, $k=1,2$ are chosen positive smooth functions such that $\sigma(x, r^\prime,t)$ satisfies the conditions immediately following \eqref{ad7} viz A1), A2) and A3). In particular, we choose $a_1$ and $a_2$ to satisfy conditions A1) and A2), for some possibly different choice of constants $\alpha, \beta, L$. And for some positive integer $M$, we define $\zeta_j:\mathbb{R}\to \mathbb{R}$ to be $1-$periodic functions,
\begin{eqnarray*} 
\zeta_1(t)&=&\phi_M\Big(2Mt\Big),\quad 0\leq t<1,\\
\zeta_2(t)&=&\phi_M\Big(2M(t-\frac 12)\Big),\quad 0\leq t<1.
\end{eqnarray*} 
where $\phi_M$ is as defined in \eqref{phim}.

To this end, we introduce a new parameter $\eta>0$
and solve for each $(x,t)$ the parameters $a^1(x,t), a^2(x,t)$ from the following expressions of the equations
for the harmonic and arithmetic averages for $1<R<2$.
\begin{eqnarray*} 
& &\int_0^1 [1+a^1(x,t)\zeta_1(r^\prime)-a^2(x,t)\zeta_2(r^\prime)]^{-2}dr^\prime\\
& &\ \ =
\left\{\begin{array}{cl}
2R^{-2}(R-1)^2(1-\phi(\frac{R-r}{\eta}))+ \psi(x,t) \phi(\frac{R-r}{\eta}),&\hbox{if } |x|<R,\\
2r^{-2}(r-1)^2,&\hbox{if }  R<|x|<2,\\
1,&\hbox{if } |x|>2,\end{array}\right.\\
& &
\int_0^1    [1+a^1(x,t)\zeta_1(r^\prime)-a^2(x,t)\zeta_2(r^\prime)]^2 dr^\prime
\\
& &\ \ =
\left\{\begin{array}{cl}
2(1-\phi(\frac{R-r}\eta))+ \psi(x,t) \phi(\frac{R-r}\eta), &\hbox{if } |x|<R,\\
2, &\hbox{if } R<|x|<2,\\
1,&\hbox{if } |x|>2,\end{array}\right.
\end{eqnarray*} 
where, $\psi(x,t)\in C^\infty({B(0,2)}\times\mathbb{R})\geq C >0$.

Thus by obtaining $a^1(x,t)=a^1_{R,\eta}(x,t)$ and $a^2(x,t)=a^2_{R,\eta}(x,t)$, the homogenized conductivity becomes
\[
\widehat\sigma(x,t)=\sigma_{R,\eta}(x,t)=
\left\{\begin{array}{cl}
\pi_{R}(1-\phi(\frac {R-r}\eta))+ \psi(x,t) \phi(\frac {R-r}\eta),&\hbox{if } |x|<R,\\
\pi_R,&\hbox{if } R<|x|<2,\\
1,&\hbox{if } |x|>2,\end{array}\right.
\]
where
\[
\pi_{R}=2R^{-2}(R-1)^2\Pi(x)+2(1-\Pi(x)).
\]
Note that the term $\psi(x,t)\phi(\frac{R-r}\eta)$
connects the exterior conductivity smoothly to the interior conductivity $\psi(x,t)$.

Now we first let $\epsilon\to 0$, then $\eta\to 0$ and finally $R\to 1$,
the obtained homogenized conductivities approximate better and better the cloaking conductivity $\sigma_{A}$ (cf \eqref{defnsigma}).
Thus 
we choose appropriate sequences $R_n\to 1$, $\eta_n\to 0$ and $\e_n\to 0$ and denote
\[
\sigma_n(x,t):=\left[1+a^1_{R_n,\eta_n}(x,t)\zeta_1(\frac r{\e_n})-a^2_{R_n,\eta_n}(x,t)
\zeta_2(\frac r{\e_n})\right]^2,\quad r=|x|. 
\]
Let $\Omega^\prime = B_3$. The above sequence $\sigma_n(x,t)$ is the desired regular isotropic sequence which approximates the cloaking for quasilinear problem in the following sense: Let $u_n\in H^1(\Omega^\prime)$ solve
\begin{equation*}\begin{aligned}
-div( \sigma_n(x,u_n)\nabla u_n(x)) &= 0 \quad\mbox{in }\Omega^\prime\\
u_n &= f \quad\mbox{on }\partial\Omega^\prime
\end{aligned}\end{equation*}
with $f\in H^{\frac{1}{2}}(\partial\Omega^\prime)$, then as $n\to\infty$
\begin{equation}\label{ad22}
 u_n \mbox{ weakly converges to }u \mbox{ in }H^1(\Omega^\prime)
\end{equation}
where $u \in H^1(\Omega^\prime)$ solves
\begin{equation*}\begin{aligned}
-div (\sigma_A(x,u)\nabla u(x)) &= 0 \quad\mbox{in }\Omega^\prime\\
u &= f \quad\mbox{on }\partial\Omega^\prime
\end{aligned}\end{equation*}
and $\sigma_A(x,t)$ is as defined in  \eqref{defnsigma} in $\Omega$ and we assume $\sigma_A(x,t) = I \text{ on } \Omega^\prime \setminus \Omega \times \mathbb{R}$.
\subsection{Convergence of DN map}
Let  $\Omega^\prime= B(0,3)$. Recall that we extend our isotropic regular cloaks $\sigma_n(x,t)$ and the perfect cloak $\sigma_A(x,t)$ by $I$ outside $\Omega = B_2$. In particular, for some $0<\kappa<1$
\begin{equation*}
 \sigma_n(x,t) =\sigma_A(x,t)= I \quad\forall (x,t)\in B(0,2+\kappa,3)\times \mathbb{R}. 
\end{equation*}
We recall the Dirichlet-to-Neumann map is now given by 
\begin{equation*}
 \Lambda_n, \Lambda : H^{\frac{3}{2}}(\partial\Omega^\prime) \to H^{\frac{1}{2}}(\partial\Omega^\prime)\quad \mbox{defined as } \Lambda_n (f) := \frac{\partial u_n}{\partial \nu} \mbox{ and } \Lambda(f) := \frac{\partial u}{\partial \nu}\mbox{ respectively}.
\end{equation*}
We then consider the state $w_n=(u_n-u)$, which satisfy 
\begin{equation*}\begin{aligned}
 -\Delta w_n &= 0 \quad\mbox{ in }B(0,2+\kappa,3)\\
 w_n &= 0 \quad\mbox{on }\partial B(0,3).
\end{aligned}\end{equation*}
From the elliptic regularity theory, we conclude $w_n\in H^2(\widetilde{\Omega})$ where \\$\widetilde{\Omega} \subset B(0,2+\kappa,3))\cup \partial B(0,3)$.  Moreover, from \cite[Theorem 9.13]{GT}, we have
\begin{equation*}
||w_n||_{H^2(\widetilde{\Omega})}\leq C||w_n||_{L^2(B(0,2+\kappa,3))}.
\end{equation*}
Now if $f\in H^{\frac{3}{2}}(\partial \Omega^\prime)$ then 
\begin{equation}\label{ad23}
 ||\frac{\partial u_n}{\partial \nu}-\frac{\partial u}{\partial \nu}||_{H^{\frac{1}{2}}(\partial\Omega^\prime)}\leq ||u_n-u||_{H^2(\widetilde{\Omega})}\leq C||u_n-u||_{L^2(B(0,2+\kappa,3))}
\end{equation}
As we see from \eqref{ad22} by using the Rellich compactness theorem we have 
\begin{equation}
 u_n \mbox{ strongly converges to }u \mbox{ in }L^2(\Omega^\prime).
\end{equation}
Thus from \eqref{ad23}, for $f\in H^{\frac{3}{2}}(\partial\Omega^\prime)$ we obtain
\begin{equation*}
||(\Lambda_n-\Lambda)(f)||_{H^{\frac{1}{2}}(\partial\Omega^\prime)} = ||\frac{\partial u_n}{\partial \nu}-\frac{\partial u}{\partial \nu}||_{H^{\frac{1}{2}}(\partial\Omega^\prime)} \to 0  \quad\mbox{ as }n\to \infty
\end{equation*}
which gives the required strong convergence of the DN map.

 \begin{appendix}
 \appendix \section[A] \\
 In this section, we show the well posedness of the boundary value problem \eqref{basiceqn} and define the Dirichlet-to-Neumann map associated to \eqref{basiceqn} in a weak sense. 
\begin{theorem}[Existence and Uniqueness]\label{appendix}
Consider $\Omega \subset \mathbb{R}^N$, $N \geq 2$, a bounded open set with Lipschitz boundary. Let $\mathcal{M}(\alpha, \beta,L;\Omega \times \mathbb{R})$ with
$0<\alpha<\beta < \infty$ and $L>0$ denote the set of all real $N\times N$ symmetric matrices $A(x,t)$ of functions defined almost
everywhere on $\Omega \times \mathbb{R}$ such that if $A(x,t)=[a_{kl}(x,t)]_{1\leq k,l\leq N}$ then 
\begin{enumerate}
 \item $a_{kl}(x,t)=a_{lk}(x,t)\ \forall l, k=1,..,N $
 \item $ (A(x,t)\xi,\xi)\geq \alpha|\xi|^2,\ |A(x,t)\xi|\leq \beta|\xi|,\ \ \forall\xi \in \mathbb{R}^N,\ \mbox{ a.e. }x\in\Omega \mbox{ and, }$
\item $|a_{kl}(x,t) - a_{kl}(x,s)| \leq L|t-s|\mbox{ for  a.e } x \in \Omega \mbox{ and any } t, s \mbox{ in } \mathbb{R}.$
\end{enumerate}

Then the boundary value problem \eqref{basiceqn} has a unique solution $u_{f} \in H^{1}(\Omega)$ for any $f \in H^{\frac{1}{2}}(\partial \Omega)$ with $||u_{f}||_{H^{1}(\Omega)} \leq C ||f||_{H^{1/2}(\partial \Omega)}$ where $C = C(\Omega, N, L,\alpha,\beta)$.
\end{theorem}
\begin{proof}
Let us fix $f \in H^{1/2}(\partial \Omega)$ and $\hat{v} \in L^{2}(\Omega)$. We consider the linear boundary value  problem 
\begin{equation}\label{app2}
\begin{aligned}
-div( A(x,\hat{v}) \nabla u) &= 0 \mbox{ in } \Omega  \\
u &= f \mbox{ on } \partial \Omega
\end{aligned}
\end{equation}
Since $A(x,\hat v) \in L^{\infty}(\Omega)$ we use Lax Milgram theorem and conclude that there exists a unique $u$ that solves the linearized boundary value problem \eqref{app2} with $||u||_{H^{1}(\Omega)} \leq C ||f||_{H^{1/2}(\partial \Omega)}$ where $C$ depends only on the parameters mentioned in the statement of the theorem.

We consider an operator $T: L^{2}(\Omega)\to L^{2}(\Omega)$ defined by $T(\hat{v}) = u$ where $u$ is a solution to (\ref{app2}). By existence and uniqueness theory for linear elliptic equations, T is well defined.  Note that $T$ maps 
$S = \left(u \in H^{1}(\Omega): ||u||_{H^{1}(\Omega)} \leq C ||f||_{H^{1/2}(\partial \Omega)}\right)$ to itself. $C$ may change line by line  but henceforth will depend only on $(\Omega, \alpha,\beta, L)$.

The existence of a solution to \eqref{basiceqn} will be proved if we show $T$ has a fixed point in S. We will use Schauder fixed point theorem to to prove this. For that, we need to show $T$ is a continuous operator and that $S$ is closed and convex subset of $L^{2}$.

Let $\hat{v_n} \to \hat{v}$ in $L^{2}(\Omega)$. Let $T(\hat{v_n}) = u_n$ and $T(\hat{v}) = u$. Since $(u_n - u) \in H^{1}_{0}(\Omega)$ we get
\begin{align}\nonumber
&\int \limits_{\Omega} \, (A(x,\hat{v_n})\nabla(u_n - u)) \cdot \nabla (u_n - u) \,dx \notag \\
&= \int \limits_{\Omega}\,( (A(x,\hat{v}) - \sigma(x,\hat{v_n}))\nabla u) \cdot \nabla (u_n - u) \,dx.
\end{align}
From Poincare inequality, Holder inequality and the uniform ellipticity assumption on $A(x,t)$ we get  
\begin{equation}\label{app3}
 \alpha ||u_n - u||_{L^{2}(\Omega)} \leq C||(A(x,\hat{v_n}) - A(x,\hat{v})) \nabla u||_{L^{2}(\Omega)}.
\end{equation}

Let $\epsilon$ be a positive real number. Let $E_n = (x \in \Omega: |\hat{v_n(x)} - \hat{v(x)}| \geq \sqrt{\epsilon}$. Since $ \hat{v_n} \to \hat {v}$ in $L^{2}(\Omega)$, $meas(E_n) \to 0$ as $n \to \infty$. (Convergence in $L^2(\Omega)$ implies  a.e convergence along a subsequence. We rename this subsequence $\hat v_n$ and proceed.) Moreover, there exists a $\delta > 0$ such that for any $E \subset \Omega$ with $ meas(E) < \delta$, we have $\int \limits_{E} \, |\nabla u|^{2} \,dx < \epsilon$. Choose a $n_0$ so that for $n \geq n_0$, $meas(E_n) < \delta$

For $n > n_0$, we then have
\begin{align*}
&\int \limits_{\Omega \setminus E_n} \, |(A(x,\hat{v}) - A(x,\hat{v_n})) \nabla u |^{2} \leq \epsilon  L ||f||^{2}_{H^{\frac{1}{2}}(\partial \Omega)} \mbox{ and }\notag \\
& \int \limits_{E_n} |(A(x,\hat{v}) - A(x,\hat{v_n})) \nabla u |^{2} \leq C \epsilon.
\end{align*}

Since $f$ is fixed and $\epsilon$ is arbitrary, $||(A(x,\hat{v_n}) - A(x,\hat{v}))\nabla u||_{L^{2}(\Omega)} \to 0$ as $n \to \infty$. By (\ref{app3}), $u_n \to u$ in $L^2(\Omega)$ proving the continuity of T. (Actually this only shows that there is a subsequence of $u_n$ that converges in $L^2(\Omega)$ to $u$. We run the entire argument with any arbitrary subsequence of $u_n$ and thereby get by the above argument that every subsequence of $u_n$ has a subsequence that converges in $L^2(\Omega)$ to $u$. Hence we can say $u_n \to u $ in $L^2(\Omega)$.) Note that $T$ is a compact operator  in $L^{2}$ topology by Rellich compactness theorem.

S is clearly a convex subset of $L^{2}(\Omega)$. To show S is closed, we let  $u_n \to u$ in $L^{2}(\Omega)$ and $u_n \in S$. The  sequence $u_n$ is bounded  in $H^{1}(\Omega)$. By Banach Alouglu theorem there exists a subsequence ${u_n}_{k}$ so that ${u_n}_{k} \to w$ weakly in $H^{1}(\Omega)$. We also have 
\begin{align*}
||w||_{H^{1}(\Omega)} \leq \liminf \limits _{k\to\infty}\;||{u_n}_k||_{H^{1}(\Omega)} \leq C||f||_{H^{1/2}(\partial \Omega)}
\end{align*}
Weak convergence in $H^{1}$ implies weak convergence in $L^{2}$. Thus by uniqueness of limits, $u=w$ showing S is closed. 

Let us now prove uniqueness. Let there exist $2$ solutions $u_1$ and $u_2$ to \eqref{basiceqn}. For $\epsilon >0$ define $F_{\epsilon}$ as 
\[F_{\epsilon}(x)= \begin{cases} 
      \int_{0}^{\epsilon} \frac{dt}{L^2t^2} & x \geq \epsilon\\
      0 & x \leq \epsilon \\
   \end{cases}
\]
$F_{\epsilon}$ is a piecewise continuous $C^1$ function such that
\begin{equation*}
|F_{\epsilon}^\prime(x)| \leq \frac{1}{L^2 \epsilon^2} \quad \forall x \in \mathbb{R}
\end{equation*}
By Corollary 2.14 in \cite{Chipot},
\[
F_{\epsilon}(u_1 - u_2) \in H^1_0(\Omega)
\]
We thus have, for any $v \in H^1_0(\Omega)$,
\[
\int_{\Omega} A(x,u_1) \nabla(u_1 - u_2) \cdot \nabla v \,dx = \int_{\Omega} (A(x,u_2) - \sigma(x,u_1)) \nabla u_2 \cdot\nabla v \,dx.
\]
Choose $v = F_{\epsilon}(u_1 - u_2)$ to get
\begin{align}\label{1143}
&\int_{\{u_1 - u_2 > \epsilon \}} A(x,u_1) \frac{|\nabla(u_1 - u_2)|^2}{L^2 (u_1 - u_2)^2}\,dx \notag \\
&=\int_{\{u_1 - u_2 > \epsilon \}}\frac{(A(x,u_2) - A(x,u_1)) \nabla u_2 \cdot \nabla(u_1 - u_2)}{L^2 (u_1 - u_2)^2}\,dx.
\end{align}
From \eqref{1143}, the fact that $t \mapsto A(.,t)$ is uniformly Lipschitz  and Cauchy-Schwartz inequality, we obtain
\begin{align*}
&\alpha \int_{\{u_1 - u_2 > \epsilon \}} \frac{|\nabla(u_1 - u_2)|^2}{L^2 (u_1 - u_2)^2} \notag \\
&\leq \int_{\{u_1 - u_2 > \epsilon \}}\frac{|(A(x,u_2) - A(x,u_1)|) |\nabla u_2|  |\nabla(u_1 - u_2)|}{L^2 (u_1 - u_2)^2}\,dx \notag \\
&\leq \int_{\{u_1 - u_2 > \epsilon \}}\frac{ |\nabla u_2|  |\nabla(u_1 - u_2)|}{L |(u_1 - u_2)|}\,dx \notag \\
&\leq \left[ \int_{\{u_1 - u_2 > \epsilon \}} \frac{|\nabla(u_1 - u_2)|^2}{L^2 (u_1 - u_2)^2} \right]^{1/2} \left[ \int_{\{u_1 - u_2 > \epsilon \}} |\nabla u_2|^2\,dx \right]^{1/2}.
\end{align*}
From this it follows that
\begin{align}\label{1144}
&\left[ \int_{\{u_1 - u_2 > \epsilon \}} \frac{|\nabla(u_1 - u_2)|^2}{L^2 (u_1 - u_2)^2} \right] \notag \\
&\leq \frac{1}{\alpha^2} \int_{\{u_1 - u_2 > \epsilon \}} |\nabla u_2|^2 \,dx \notag \\
&\leq \frac{1}{\alpha^2} \int_{\Omega} |\nabla u_2|^2 \,dx.
\end{align}
Set 
\[G_{\epsilon}(x)= \begin{cases} 
      \int_{0}^{\epsilon} \frac{dt}{Lt} & x \geq \epsilon\\
      0 & x \leq \epsilon \\
   \end{cases}
\]
With this definition, \eqref{1144} becomes
\begin{align}\label{1146}
\int_{\Omega} |\nabla G_\epsilon (u_1 - u_2)|^2\,dx \leq \frac{1}{\alpha^2} \int_\Omega |\nabla u_2|^2 \,dx.
\end{align}
Note that $G_\epsilon$ is a continuous piecewise $C^1$ function satisfying the assumptions of Corollary 2.15 in \cite{Chipot}. Thus by Corollary 2.15 in \cite{Chipot} we obtain
\begin{align}
G_\epsilon(u_1 - u_2) \in H^1_0(\Omega)
\end{align}
Use Poincare inequality in \eqref{1146} to get
\begin{align*}
\int_\Omega |G_{\epsilon}(u_1 - u_2)^2\,dx \leq C \int_\Omega |\nabla u_2|^2\,dx.
\end{align*}
We now pass through the limit in $\epsilon$ and use Fatou's lemma to get
\begin{equation*}
\int \limits_{\Omega} \liminf_{\epsilon \to 0} |G_{\epsilon}(u_1 - u_2)^2\,dx \leq C \int_\Omega |\nabla u_2|^2\,dx.
\end{equation*}
Hence 
\begin{equation*}
\liminf_{\epsilon \to 0} |G_{\epsilon}(u_1 - u_2)^2 < + \infty \mbox{ a.e }x \in \Omega.
\end{equation*}
The definition of $G_\epsilon$ then implies that
\begin{align*}
u_1 - u_2 \leq 0 \mbox{ a.e } x \in \Omega.
\end{align*}
Switch the roles of $u_1$ and $u_2$ to conclude that 
\begin{align*}
u_1 = u_2 \mbox{ a.e }x \in \Omega.
\end{align*}
\end{proof}
Theorem \ref{appendix} allows us to define the Dirichlet-to-Neumann map weakly. The DN map
\[
\Lambda_{\sigma}:H^{1/2}(\partial \Omega) \to H^{-1/2}(\partial \Omega)
\]
is defined weakly as 
\[
\langle \Lambda_{A}(f),g\rangle = \int \limits_{\Omega}A(x,u) \nabla u \cdot\nabla v \,dx
\]
where $u$ is the unique solution to \eqref{basiceqn} and $v$ is any $H^{1}(\Omega)$ function with trace $g$. This is the natural generalization of the Dirichlet-to-Neumann map to a quasi-linear equation of divergence type.

We end this section by stating a result on higher global regularity of solutions to the  quasilinear elliptic equation \eqref{basiceqn}.
\begin{lemma}[Regularity]\label{ub4}
Let $\Omega \subset \mathbb{R}^N$, $N \geq 2$ be a bounded domain with $C^{2,\gamma}$ boundary, $0<\gamma<1$. Let $A(x,t) = [a_{ij}(x,t)]_{N \times N}\in \mathcal{M}(\alpha,\beta,L;\overline{\Omega}\times\mathbb{R})$ and we further assume that $A \in C^{1,\gamma}(\overline{\Omega}  \times \mathbb{R})$, then the quasilinear boundary va
lue problem
\begin{align}\label{higherregularity}
\nabla \cdot A(x,u) \nabla u &= 0 \mbox{ in } \Omega \notag \\
 u|_{\partial \Omega} &= f 
\end{align}
has a unique solution $u \in C^{2,\gamma}(\overline{ \Omega})$ for any $f \in C^{2,\gamma}(\overline{ \Omega})$ satisfying for all $x,y\in\overline{\Omega}$ 
\[|u(x) - u(y)| \leq C_f |x-y|^\lambda\]
where $C_f=C(N,\Omega,\alpha,\beta,f)$ and $\lambda=\lambda(N,\Omega,\alpha,\beta,\gamma)$. \end{lemma}
\begin{proof}
The existence and uniqueness result for the quasilinear boundary value problem considered in \eqref{higherregularity} can be found in  \cite{GT} or \cite{LYZ}. 

For proving the H\"older estimates on $u$, we let $v = u - f$ and consider the linear boundary value problem 
\begin{align*}
\nabla \cdot A(x,u)\nabla v &= - \nabla \cdot A(x,u) \nabla f \notag \\
v|_{\partial \Omega} &= 0
\end{align*}
Let $\varphi = -\nabla \cdot A(x,u) \nabla f $ then we note that $\varphi\in C^{\gamma}(\overline{\Omega})$. By considering $\varphi \in L^p(\Omega)$ for some $p>\frac{n}{2}$ we obtain, for all $x,y\in \overline{\Omega}$ 
\[
|v(x) - v(y)| \leq C ||\varphi||_{L^p(\Omega)}|x-y|^{\lambda^{\prime}}
\]
where $C=C(N,\Omega,\alpha,\beta)$ and $\lambda^{\prime}=\lambda^{\prime}(N,\Omega,\alpha,\beta)$. Now by using the triangle inequality on $u=v+f$ we conclude that $u$ is also H\"{o}lder continuous with the exponent $\lambda = min\{\lambda^{\prime},\gamma\}$. 
\end{proof}
\end{appendix}

\bibliographystyle{plain}
\bibliography{Master_bibfile}
\end{document}